\documentclass[final,reqno]{siamltex}
\usepackage{latexsym,amsmath,amssymb,amsfonts,mathrsfs}
\usepackage{epsf,graphicx,epsfig,color,cite,cases,epstopdf}
\usepackage{subfigure,graphics,multirow,marginnote,bm}
\newcommand{\R}{{\mathbb R}}
\newcommand{\Z}{{\mathbb Z}}
\newcommand{\N}{{\mathbb N}}

\newcommand{\Sp}{{\mathbb S}}
\newcommand{\ds}{\displaystyle}
\newcommand{\no}{\nonumber}
\newcommand{\be}{\begin{eqnarray}}
\newcommand{\ben}{\begin{eqnarray*}}
\newcommand{\en}{\end{eqnarray}}
\newcommand{\enn}{\end{eqnarray*}}
\newcommand{\ba}{\backslash}
\newcommand{\pa}{\partial}

\newcommand{\ov}{\overline}
\newcommand{\I}{{\rm Im}}
\newcommand{\Rt}{{\rm Re}}

\newcommand{\G}{\Gamma}

\newcommand{\vep}{\varepsilon}

\newcommand{\om}{\omega}

\newcommand{\wid}{\widetilde}

\newcommand{\se}{\setminus}

\newcommand{\tb}{\textbf}
\newcommand{\ra}{\rightarrow}
\definecolor{hw}{rgb}{0,0,0}
\newtheorem{remark}[theorem]{Remark}

\begin{document}
\renewcommand{\theequation}{\arabic{section}.\arabic{equation}}

\title{\bf Uniqueness in inverse acoustic and electromagnetic scattering with phaseless near-field data
at a fixed frequency
}
\author{Xiaoxu Xu\thanks{Academy of Mathematics and Systems Science, Chinese Academy of Sciences,
Beijing 100190, China and School of Mathematical Sciences, University of Chinese
Academy of Sciences, Beijing 100049, China ({\tt xuxiaoxu14@mails.ucas.ac.cn})}
\and
Bo Zhang\thanks{LSEC, NCMIS and Academy of Mathematics and Systems Science, Chinese Academy of
Sciences, Beijing, 100190, China and School of Mathematical Sciences, University of Chinese
Academy of Sciences, Beijing 100049, China ({\tt b.zhang@amt.ac.cn})}
\and
Haiwen Zhang\thanks{NCMIS and Academy of Mathematics and Systems Science, Chinese Academy of Sciences,
Beijing 100190, China ({\tt zhanghaiwen@amss.ac.cn})}
}
\date{}

\maketitle

\begin{abstract}
This paper is concerned with uniqueness results in inverse acoustic and electromagnetic scattering problems
with phaseless total-field data at a fixed frequency.
Motivated by our previous work ({\em SIAM J. Appl. Math. \bf78} (2018), 1737-1753), 
where uniqueness results were proved for inverse acoustic scattering with phaseless far-field data 
generated by superpositions of two plane waves as the incident waves at a fixed frequency, 
in this paper, we use superpositions of two point sources as the incident fields at a fixed frequency 
and measure the modulus of the acoustic total-field (called phaseless acoustic near-field data) 
on two spheres enclosing the scatterers generated by such incident fields on the two spheres.
Based on this idea, we prove that the impenetrable bounded obstacle or the index of refraction of 
an inhomogeneous medium can be uniquely determined from the phaseless acoustic near-field data 
at a fixed frequency. 
Moreover, the idea is also extended to the electromagnetic case, and it is proved that 
the impenetrable bounded obstacle or the index of refraction of an inhomogeneous medium can be 
uniquely determined by the phaseless electric near-field data at a fixed frequency, that is,  
the modulus of the tangential component with the orientations $\bm e_\phi$ and $\bm e_\theta$, respectively, 
of the electric total-field measured on a sphere enclosing the scatters and generated by 
superpositions of two electric dipoles at a fixed frequency located on the measurement sphere and 
another bigger sphere with the polarization vectors $\bm e_\phi$ and $\bm e_\theta$, respectively.
As far as we know, this is the first uniqueness result for three-dimensional inverse electromagnetic 
scattering with phaseless near-field data. 
\end{abstract}



\pagestyle{myheadings}
\thispagestyle{plain}
\markboth{X. Xu, B. Zhang, and H. Zhang}{Uniqueness in inverse scattering with
phaseless near-field data at a fixed frequency}

\section{Introduction}\label{sec1}

Inverse scattering problems occur in many applications such as radar, remote sensing, geophysics, 
medical imaging and nondestructive testing. These problems aim at reconstructing the unknown 
scatterers from the measurement data of the scattered waves. 
In the past decades, inverse acoustic and electromagnetic scattering problems with phased data have
been extensively studied mathematically and numerically. 
A comprehensive account of these studies can be found in the monographs \cite{ChenXD18,CK}.

In many practical applications, it is much harder to obtain data with accurate phase information
compared with just measuring the intensity (or the modulus) of the data, and thus it is
often desirable to study inverse scattering with phaseless data 
(see, e.g., \cite[Chapter 8]{ChenXD18} and the references quoted there).
In fact, inverse scattering problems with phaseless data have also been widely studied 
numerically over the past decades (see, e.g.
\cite{BaoLiLv2012,BaoZhang16,ChenXD18,CH16,CH17,CFH17,IK2011,KR16,Pan11,MOTL97,S16,ZZ01,ZZ02,ZZ03}
and the references quoted there).

Recently, uniqueness and stability results have also been established for inverse scattering with phaseless 
data (see, e.g. \cite{ACZ16,JLZ19a,JLZ19b,K14,K17,K17x,MH17,N15,N16,Yamamoto18,XZZ,XZZ2,XZZ3,ZG18}). 
For example, for point source incidence uniqueness results have been established in \cite{K14,K17} 
for inverse potential and acoustic medium scattering with the phaseless near-field data
generated by point sources placed on a sphere enclosing the scatterer and measured in a small ball 
centered at each source position for an interval of frequencies, and in \cite{N15} for inverse acoustic 
medium scattering with the phaseless near-field data measured on an annulus surrounding the scatterer 
at fixed frequency.

The purpose of this paper is to propose a new approach to establish uniqueness results for inverse 
acoustic scattering problems with phaseless total-field data at a fixed frequency. 
Motivated by our previous work \cite{XZZ}, where uniqueness results have been proved for inverse 
acoustic scattering with phaseless far-field data corresponding to superpositions of two plane waves 
as the incident fields at a fixed frequency, we consider to utilize the superposition of two point 
sources at a fixed frequency as the incident field.
However, the idea of proofs used in \cite{XZZ} can not be applied directly to the inverse
scattering problem with phaseless near-field data. This is due to the fact that our proofs
in \cite{XZZ} are based essentially on the limit of the normalized eigenvalues of the
far-field operators. To overcome this difficulty, we consider to use two spheres, which enclose 
the scatterers, as the locations of such incident fields and the measurement surfaces of the modulus 
of the acoustic total-field (the sum of the incident field and the scattered field).
In fact, many phase retrieval algorithms have been developed for inverse scattering problems with 
phaseless near-field data measured on two surfaces to ensure the reliability of the near-field 
phase reconstruction algorithms (see, e.g. \cite{I3E1,I3E3,I3E6}). 
Based on this idea, we prove that the impenetrable bounded obstacle or the index of refraction of the 
inhomogeneous medium can be uniquely determined from the phaseless total-field data at a fixed frequency. 
Note that the superposition of two point sources was also used in \cite{Yamamoto18} 
as the incident field to study uniqueness for phaseless inverse scattering problems.
Some related uniqueness results can be found in \cite{ZSGL19,ZWGL19}. 

The idea is also applied to phaseless inverse electromagnetic scattering which is more complicated 
than the acoustic case. In this case, the electric total field is a complex vector-valued function, 
so we need to define the phaseless data used in this paper.
In many applications (see, e.g. \cite{Brown,Pan11,Schmidt}), the phaseless near-field data 
are based on the measurement of the modulus of the tangential component of the electric total field 
on the measurement surface. Further, it has been elaborated in \cite{Hansen} that the measurement 
data are based on two tangential components of the electric field on the measurement sphere 
(see \cite[p.100]{Hansen}). Therefore, the phaseless near-field data used is 
the modulus of the tangential component in the orientations $\bm e_\phi$ and $\bm e_\theta$, respectively, 
of the electric total field measured on a sphere enclosing the scatters and generated 
by superpositions of two electric dipoles at a fixed frequency located on the measurement sphere 
and another bigger sphere with the polarizations $\bm e_\phi$ and $\bm e_\theta$, respectively.
Following a similar idea as in the acoustic case, we prove that the impenetrable bounded obstacle or
the refractive index of the inhomogeneous medium (under the condition that the magnetic permeability is
a positive constant) can be uniquely determined by the phaseless total-field data at a fixed frequency.
To the best of our knowledge, this is the first uniqueness result for three-dimensional inverse 
electromagnetic scattering with phaseless near-field data. 
It should be mentioned that our uniqueness results in this paper are based on parts of the PhD thesis \cite{Xu19}. 

The outline of this paper is as follows. The acoustic and electromagnetic scattering models considered 
are given in Section \ref{direct}. Sections \ref{acoustic} and \ref{em} are devoted to the uniqueness 
results for phaseless inverse acoustic and electromagnetic scattering problems, respectively.
Conclusions are given in Section \ref{con}.

\section{The direct scattering problems}\label{direct}
\setcounter{equation}{0}

We will introduce the acoustic and electromagnetic scattering models considered in this paper.
To this end, assume that $D$ is an open and bounded domain in $\R^3$ with a $C^2-$boundary $\pa D$
such that the exterior $\R^3\ba\ov D$ is connected. Assume further that $\ov{D}\subset B_{R_1}$,
where $B_{R_1}$ is a ball centered at the origin with radius $R_1>0$ large enough.

\subsection{The acoustic case}\label{s2.1}

In this paper, we consider the problem of acoustic scattering by an impenetrable obstacle or
an inhomogeneous medium in $\R^3$. We need the following fundamental solution to the three-dimensional
Helmholtz equation $\Delta w+k^2w=0$ in $\R^3$ with $k>0$:
\ben
\Phi_k(x,y):=\frac{e^{ik|x-y|}}{4\pi|x-y|},\;\;\;x,y\in\R^3,\;\;x\not=y.
\enn

For arbitrarily fixed $y\in\R^3\ba\ov{D}$ consider the time-harmonic ($e^{-i\omega t}$ time dependence)
point source
\ben
w^i:=w^i(x,y)=\Phi_k(x,y),\quad x\in\R^3\ba\ov{D},
\enn
which is incident on the obstacle $D$ from the unbounded part $\R^3\ba\ov{D}$,
where $k=\omega/c>0$ is the wave number, $\omega$ and $c$ are the wave frequency and speed in
the homogeneous medium in the whole space. Then the problem of scattering of the point source $w^i$
by the impenetrable obstacle $D$ is formulated as the exterior boundary value problem:
\be\label{he}
\Delta_xw^s(x,y)+k^2w^s(x,y)=0,&&\quad\;x\in\R^3\ba\ov{D},\\ \label{bc}
\mathscr Bw=0&&\quad\text{on}\;\;\pa D,\\ \label{rc}
\lim_{r\rightarrow\infty}r\left(\frac{\partial w^s}{\partial r}-ikw^s\right)=0,&&\quad r=|x|,
\en
where $w^s$ is the scattered field, $w:=w^i+w^s$ is the total field, and (\ref{rc}) is
the Sommerfeld radiation condition imposed on the scattered field $w^s$.
The boundary condition $\mathscr B$ in (\ref{bc}) depends on the physical property of the obstacles $D$:
\ben\left\{\begin{array}{ll}
{\mathscr B}w:=w\;\;\text{on}\;\;\pa D & \text{if $D$ is a sound-soft obstacle},\\
{\mathscr B}w:={\pa w}/{\pa\nu}+\eta w\;\;\text{on}\;\;\pa D & \text{if $D$ is an impedance obstacle},\\
{\mathscr B}w:=w\;\;\text{on}\;\;\G_D,\;\;{\mathscr B}w:={\pa w}/{\pa\nu}+\eta w\;\;\text{on}\;\;\G_I &
   \text{if $D$ is a partially coated obstacle},
\end{array}\right.
\enn
where $\nu$ is the unit outward normal to the boundary $\pa D$ and $\eta$ is the impedance function on $\pa D$
satisfying that $\I[\eta(x)]\geq0$ for all $x\in\pa D$ or $x\in\G_I$.
We assume that $\eta\in C(\pa D)$ or $\eta\in C(\G_I)$, that is, $\eta$ is continuous on $\pa D$ or $\G_I$.
When $\eta=0$, the impedance boundary condition becomes the Neumann boundary condition (a sound-hard obstacle).
For a partially coated obstacle, we assume that the boundary $\pa D$ has a Lipschitz dissection
$\pa D=\G_D\cup\Pi\cup\G_I$, where $\G_D$ and $\G_I$ are disjoint, relatively open subsets of $\pa D$ and
having $\Pi$ as their common boundary in $\pa D$ (see, e.g., \cite{CCM}).

The problem of scattering of the point source $w^i$ by an inhomogeneous medium is modeled as follows:
\be\label{he-n}
\Delta_xw^s(x,y)+k^2n(x)w^s(x,y)=k^2(1-n(x))w^i(x,y),&&\quad\;x\in\R^3,\\ \label{rc-n}
\lim_{r\rightarrow\infty}r\left(\frac{\partial w^s}{\partial r}-ikw^s\right)=0,&&\quad  r=|x|,
\en
where $w^s$ is the scattered field and $n$ in (\ref{he-n}) is the refractive index characterizing
the inhomogeneous medium.
We assume that $n-1$ has compact support $\ov{D}$ and $n\in L^\infty(D)$
with $\Rt[n(x)]>0,\;\I[n(x)]\geq0$ for all $x\in D$.

The existence of a unique (variational) solution to the problems (\ref{he})-(\ref{rc}) and
(\ref{he-n})-(\ref{rc-n}) has been proved in \cite{CK,KG,CC,Kirsch}.
In particular, the scattered-field $w^s$ has the asymptotic behavior:
\ben
w^s(x,y)=\frac{e^{ik|x|}}{|x|}\left\{w^\infty(\hat{x},y)+\left(\frac{1}{|x|}\right)\right\},
\quad|x|\rightarrow\infty
\enn
uniformly for all observation directions $\hat{x}=x/|x|\in\Sp^2$, where $\Sp^2$ is the unit sphere
in $\R^3$ and $w^\infty(\hat{x},y)$ is the far-field pattern of $w^s$ which is an analytic function
of $\hat{x}\in\Sp^2$ for each $y\in\R^3\se\ov{D}$ (see, e.g., \cite[(2.13)]{CK}).

In this paper, we also consider the superposition of two point sources
\be\label{lsi}
w^i=w^i(x;y_1,y_2)=w^i(x,y_1)+w^i(x,y_2)=\Phi_k(x,y_1)+\Phi_k(x,y_2)
\en
as the incident field, where $y_1,y_2\in\R^3\se\ov{D}$ are the locations of the two point sources.
It then follows by the linear superposition principle that the corresponding scattered field
\be\label{lsp-s}
w^s(x;y_1,y_2)=w^s(x,y_1)+w^s(x,y_2)
\en
and the corresponding total field
\be\label{lsp-t}
w(x;y_1,y_2)=w(x,y_1)+w(x,y_2),
\en
where $w^s(x,y_j)$ and $w(x,y_j)$ are the scattered field and the total field corresponding
to the incident point source $w^i(x,y_j)$, respectively, $j=1,2$.

The {\em inverse acoustic obstacle (or medium) scattering problem} we consider in this paper
is to reconstruct the obstacle $D$ and its physical property (or the index of refraction $n$
of the inhomogeneous medium) from the phaseless total field $|w(x;y_1,y_2)|$
for $x,y_1,y_2$ on some spheres enclosing $D$ and the inhomogeneous medium.

\subsection{The electromagnetic case}\label{s2.2}

In this paper, we consider two electromagnetic scattering models, that is,
scattering by an impenetrable obstacle and scattering by an inhomogeneous medium.
We will consider the time-harmonic ($e^{-i\om t}$ time dependence) incident electric dipole
located at $y\in\R^3\se\ov{D}$ and described by the matrices $E^i(x,y)$ and $H^i(x,y)$ defined by
\ben
E^i(x,y)p:=\frac{i}{k}{\rm curl}_x{\rm curl}_x[p\Phi_k(x,y)],\;\;\;
H^i(x,y)p:={\rm curl}_x[p\Phi_k(x,y)],\;\;x\not=y,
\enn
for $x\in\R^3\se\ov{D}$, where $p\in\R^3$ is the polarization vector, $k:=\om/c>0$ is the wave number,
$\om$ and $c:=1/{\sqrt{\vep_0\mu_0}}$ are the wave frequency and speed in the homogeneous medium
in $\R^3\se\ov{D}$, respectively, and $\vep_0$ and $\mu_0$ are the electric permittivity and
the magnetic permeability of the homogeneous medium, respectively.
A direct calculation shows that for $x\neq y$,
\be\label{ele_ei}
E^i(x,y)&=&ik\Phi_k(x,y)I+\frac{i}{k}\nabla_x\nabla_x\Phi_k(x,y)\\ \no
&=&\frac ik\left\{\left[k^2+\left(ik-\frac1{|x-y|}\right)\frac1{|x-y|}\right]I
+\widehat{x-y}\cdot\widehat{x-y}^\top f(|x-y|)\right\}\Phi_k(x,y),
\en
where $I$ is a $3\times3$ identity matrix, $\nabla_x\nabla_x:=(\pa_{x_i}\pa_{x_j})_{3\times3}$, $\widehat{x-y}=({x-y})/{|x-y|}$ and $f(r):=3/{r^2}-{3ik}/r-k^2$.
Then the problem of scattering of the electric dipole $E^i$ and $H^i$ by the impenetrable obstacle $D$
can be modeled as the exterior boundary value problem:
\be\label{ele_e1}
{\rm curl}_xE^s-ikH^s=0& &\text{in}\;\;\R^3\se\ov{D},\\ \label{ele_e2}
{\rm curl}_xH^s+ikE^s=0& &\text{in}\;\;\R^3\se\ov{D},\\ \label{ele_bc}
\mathscr BE=0& &\text{on}\;\;\pa D,\\ \label{ele_rc}
\lim_{r\ra\infty}(H^s\times x-rE^s)=0,& &r=|x|,
\en
where $(E^s,H^s)$ is the scattered field, $E:=E^i+E^s$ and $H:=H^i+H^s$ are the electric total field
and the magnetic total field, respectively, and (\ref{ele_rc}) is the Silver--M\"uller radiation condition
which holds uniformly for all $\hat x\in\Sp^2$ and ensures the uniqueness of the scattered field.
The boundary condition $\mathscr B$ in (\ref{ele_bc}) depends on the physical property of the obstacle $D$,
that is, $\mathscr BE:=\nu\times E$ on $\pa D$ (called as the PEC condition) if $D$ is a perfect conductor,
where $\nu$ is the unit outward normal to the boundary $\pa D$,
$\mathscr BE:=\nu\times{\rm curl}E-i\lambda(\nu\times E)\times\nu$ on $\pa D$
if $D$ is an impedance obstacle, where $\lambda$ is the impedance function on $\pa D$, and
\ben
\mathscr BE:=\nu\times E\;\;\text{on}\;\;\G_D,\;\;\;
\mathscr BE:=\nu\times{\rm curl}E-i\lambda(\nu\times E)\times\nu\;\;\text{on}\;\;\G_I
\enn
if $D$ is a partially coated obstacle, where $\pa D$ has a Lipschitz dissection
$\pa D=\G_D\cup\Pi\cup\G_I$ with $\G_D$ and $\G_I$ being disjoint and relatively open subsets
of $\pa D$ and having $\Pi$ as their common boundary in $\pa D$ and $\lambda$ is the impedance
function on $\G_I$.
We assume throughout this paper that $\lambda\in C(\pa D)$ with $\lambda(x)\geq0$ for all $x\in\pa D$
or $\lambda\in C(\G_I)$ with $\lambda(x)\geq0$ for all $x\in\G_I$.

The problem of scattering of an electric dipole by an inhomogeneous medium is modeled 
as the medium scattering problem:
\be\label{ele_em1}
{\rm curl}_xE^s-ikH^s=0 &\text{in}\;\;\R^3\\ \label{ele_em2}
{\rm curl}_xH^s+ikn(x)E^s=ik(1-n(x))E^i &\text{in}\;\;\R^3\\ \label{ele_rcm}
\lim_{r\ra\infty}(H^s\times x-rE^s)=0, &r=|x|,
\en
where $(E^s,H^s)$ is the scattered field and $(E,H):=(E^i,H^i)+(E^s,H^s)$ is the total field.
The refractive index $n(x)$ in (\ref{ele_em2}) is given by
\ben
n(x):=\frac1{\vep_0}\left(\vep(x)+i\frac{\sigma(x)}\om\right).
\enn
In this paper, we assume the magnetic permeability $\mu=\mu_0$ to be a positive constant in
the whole space. We assume further that $n-1$ has a compact support $\ov{D}$
and $n\in C^{2,\gamma}(\R^3)$ for $0<\gamma<1$ with ${\rm Re}[n(x)]>0$ and ${\rm Im}[n(x)]\geq0$
for all $x\in D$.

The existence of a unique (variational) solution to the problems (\ref{ele_e1})--(\ref{ele_rc})
and (\ref{ele_em1})--(\ref{ele_rcm}) has been established in \cite{CCM,CCM2,CK}.
In particular, it is well known that the electromagnetic scattered field $E^s$
has the asymptotic behavior:
\ben
E^s(x,y)p=\frac{e^{ik|x|}}{|x|}\left\{E^\infty(\hat{x},y)p+\left(\frac{1}{|x|}\right)\right\},
\quad|x|\rightarrow\infty
\enn
uniformly for all observation directions $\hat{x}=x/|x|\in\Sp^2$, where $E^\infty(\hat{x},y)$
is the electric far-field pattern of $E^s$ which is an analytic function of $\hat{x}\in\Sp^2$
for each $y\in\R^3\se\ov{D}$ (see, e.g., \cite[(6.23)]{CK}). Because of the linearity of the
direct scattering problem with respect to the incident field, we can express the scattered
waves by matrices $E^s(x,d)$ and $H^s(x,d)$, the total waves by matrices $E(x,d)$ and $H(x,d)$,
and the far-field patterns by $E^\infty(\hat x,d)$ and $H^\infty(\hat x,d)$, respectively.

We will also consider the following superposition of two electric dipoles as the incident field:
\ben
E^i&=&E^i(x,y_1,p_1,\tau_1,y_2,p_2,\tau_2):=\tau_1E^i(x,y_1)p_1+\tau_2E^i(x,y_2)p_2\\
   &=&\frac ik{\rm curl}_x{\rm curl}_x[\tau_1p_1\Phi_k(x,y_1)+\tau_2p_2\Phi_k(x,y_2)],\\
H^i&=&H^i(x,y_1,p_1,\tau_1,y_2,p_2,\tau_2):=\tau_1H^i(x,y_1)p_1+\tau_2H^i(x,y_2)p_2\\
   &=&{\rm curl}_x[\tau_1p_1\Phi_k(x,y_1)+\tau_2p_2\Phi_k(x,y_2)],
\enn
where $x\in\R^3$, $y_1,y_2\in\R^3\se\ov{D}$, $x\neq y_1$, $x\neq y_2$, $p_1,p_2\in\R^3$
and $\tau_1,\tau_2\in\{0,1\}$. By the linear superposition principle, the electric total
field and scattered field corresponding to the superposition of two electric dipoles as
the incident field satisfy
\[E^s(x,y_1,p_1,\tau_1,y_2,p_2,\tau_2):=\tau_1E^s(x,y_1)p_1+\tau_2E^s(x,y_2)p_2\]
and
\be\label{ele_lst}
E(x,y_1,p_1,\tau_1,y_2,p_2,\tau_2):=\tau_1E(x,y_1)p_1+\tau_2E(x,y_2)p_2,
\en
where $E^s(x,y_j)p_j$ and $E(x,y_j)p_j$ are the electric scattered field and the electric
total field corresponding to the incident field $E^i(x,y_j)p_j$, respectively, $j=1,2$.

Following \cite{Hansen,Pan11,Schmidt}, we measure the modulus of the tangential component of
the electric total field on a sphere $\pa B_r$ centered at the origin with radius $r>0$.
To represent the tangential components, we introduce the following spherical coordinate
\ben
\begin{cases}
x_1=r\sin\theta\cos\phi,\\
x_2=r\sin\theta\sin\phi,\\
x_3=r\cos\theta,
\end{cases}
\enn
with $x:=(x_1,x_2,x_3)\in\R^3$ and $(r,\theta,\phi)\in[0,+\infty)\times[0,\pi]\times[0,2\pi)$.
For any $x\in\pa B_r\se\{N_r,S_r\}$, the spherical coordinate gives an one-to-one correspondence
between $x$ and $(r,\phi,\theta)$. Here, $N_r:=(0,0,r)$ and $S_r:=(0,0,-r)$ denote the north
and south poles of $\pa B_r$, respectively. If we define
\ben
\bm e_\phi(x):=(-\sin\phi,\cos\phi,0),\;\;\;\;
\bm e_\theta(x):=(\cos\theta\cos\phi,\cos\theta\sin\phi,-\sin\theta),
\enn
then $\bm e_\phi(x)$ and $\bm e_\theta(x)$ are two orthonormal tangential vectors of $\pa B_r$
at $x\notin\{N_r,S_r\}$. Now, we can represent our phaseless measurement data by
\[|\bm e_m(x)\cdot E(x,y_1,\bm e_n(y_1),\tau_1,y_2,\bm e_l(y_2),\tau_2)|\]
with $x,y_1,y_2\in\pa B_r\se\{N_r,S_r\}$, $x\neq y_1$, $x\neq y_2$, $m,n,l\in\{\phi,\theta\}$ and $\tau_1,\tau_2\in\{0,1\}$.

The \emph{inverse electromagnetic obstacle or medium scattering problem} we consider in this
paper is to reconstruct the obstacle $D$ and its physical property or the index of refraction $n$
of the inhomogeneous medium from the modulus of the tangential component of the electric total
field, $|\bm e_m(x)\cdot E(x,y_1,\bm e_n(y_1),\tau_1,y_2,\bm e_l(y_2),\tau_2)|$,
for all $x,y_1,y_2$ in some spheres enclosing $D$ or the inhomogeneous medium,
$m,n,l\in\{\phi,\theta\}$ and $\tau_1,\tau_2\in\{0,1\}$. The purpose of this paper is to
prove uniqueness results for the above inverse acoustic and electromagnetic scattering problems.

\section{Inverse acoustic scattering with phaseless total-field data}\label{acoustic}
\setcounter{equation}{0}

This section is devoted to the uniqueness results for inverse acoustic scattering with phaseless
total-field data at a fixed frequency measured on two spheres enclosing the scatterers 
(see Figures \ref{nearo} and \ref{nearm}).

Denote by $w_j^s$ and $w_j$ the scattered field and the total field,
respectively, associated with the impenetrable obstacle $D_j$ (or the refractive index $n_j$)
and corresponding to the incident field $w^i$, $j=1,2$.
Let $B_{R_2}$ denote the ball centered at the origin with
radius $R_2>R_1>0$ with $\pa B_{R_2}$ denoting the boundary of $B_{R_2}$.
By appropriately choosing $R_2>R_1>0$, it can be ensured that $k^2$ is not a Dirichlet eigenvalue
of $-\Delta$ in $B_{R_2}\se\ov{B_{R_1}}$. Here, $k^2$ is called a Dirichlet eigenvalue of
$-\Delta$ in a bounded domain $V$ if the the interior Dirichlet boundary value problem
\[\begin{cases}
\Delta u+k^2u=0 & \text{in}\;V,\\
u=0 & \text{on}\;\pa V
\end{cases}\]
has a nontrivial solution $u$. The above assumption on $R_1$ and $R_2$ can be easily satisfied
since the Dirichlet eigenvalues of $-\Delta$ in a bounded domain are discrete and satisfy the strong
monotonicity property \cite[Theorem 4.7]{L86} (see also the arguments in the proof of \cite[Theorem 5.2]{CK}).
Let $G$ denote the unbounded component of the complement of $D_1\cup D_2$.
Then we have the following global uniqueness results for the phaseless inverse scattering problems.
\begin{figure}[!ht]
\begin{minipage}[t]{0.5\linewidth}
\centering
\includegraphics[width=1\textwidth]{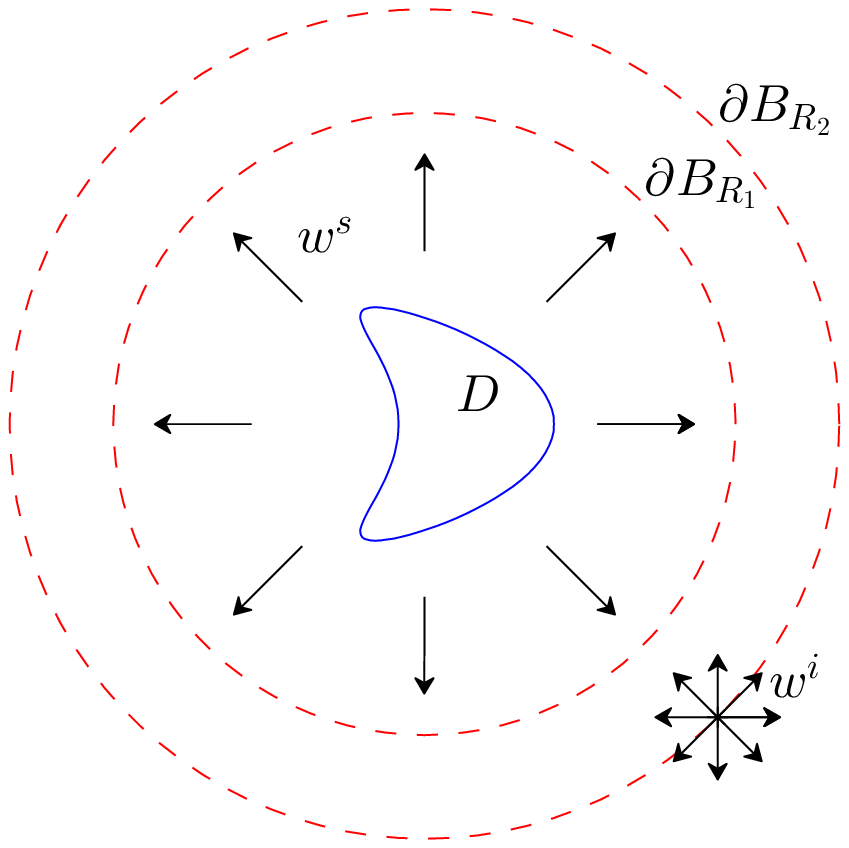}
\caption{Scattering by an obstacle.}\label{nearo}
\end{minipage}
\begin{minipage}[t]{0.5\linewidth}
\centering
\includegraphics[width=1\textwidth]{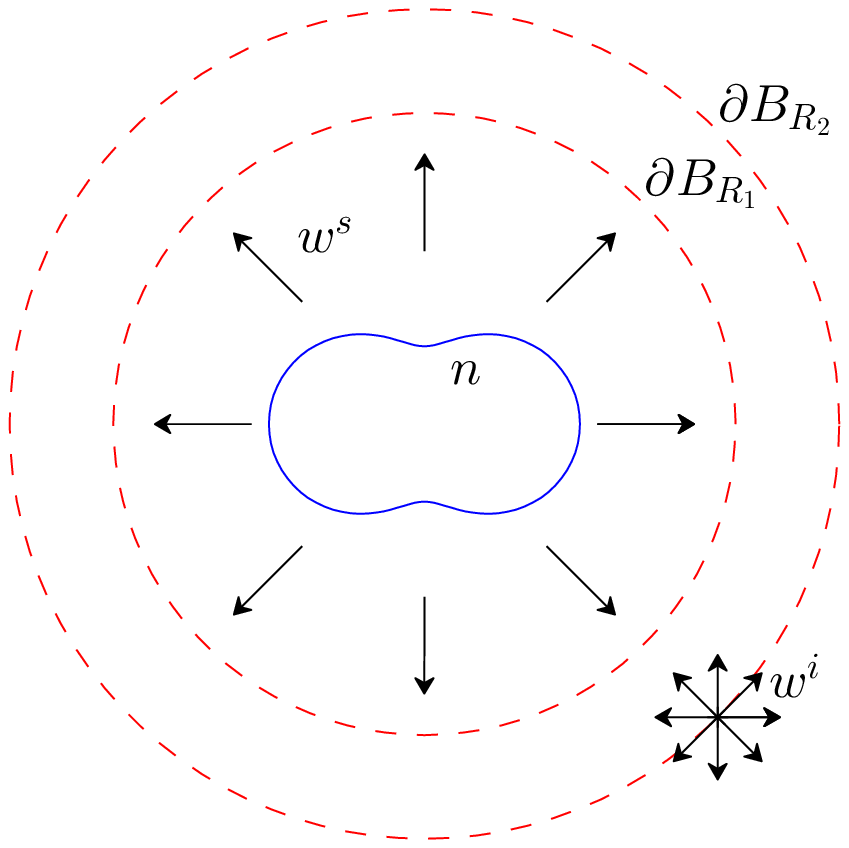}
\caption{Scattering by a medium.}\label{nearm}
\end{minipage}
\end{figure}

\begin{theorem}\label{tt}
Let $D_1,D_2$ be two bounded domains and let $R_2>R_1>0$ be large enough so that
$\ov{D_1\cup D_2}\subset B_{R_1}$. Assume that $k^2$ is not a Dirichlet eigenvalue of
$-\Delta$ in $B_{R_2}\ba\ov{B_{R_1}}$.

(a) Assume that $D_1$ and $D_2$ are two impenetrable obstacles with boundary conditions
$\mathscr B_1$ and $\mathscr B_2$, respectively. If the corresponding total fields satisfy
\be\no
|w_1(x,y)|=|w_2(x,y)|,\quad\forall(x,y)\in (\pa B_{R_1}\times \pa B_{R_1})
\cup(\pa B_{R_2}\times(\{y_0\}\cup\pa B_{R_2})),\;\;\\ \label{eq1}
x\neq y\;\;\;\quad
\en
and
\be\no
|w_1(x;y,y_0)|=|w_2(x;y,y_0)|,\;\;\forall(x,y)\in (\pa B_{R_1}\times \pa B_{R_1})
\cup(\pa B_{R_2}\times \pa B_{R_2}),\;\;\\ \label{eq2}
x\neq y,\;y_0\;\;\;
\en
for an arbitrarily fixed $y_0\in \pa B_{R_1}$, then $D_1=D_2$ and $\mathscr B_1=\mathscr B_2$.

(b) Assume that $n_1,n_2\in L^\infty(\R^3)$ are the indices of refraction of two
inhomogeneous media with $n_j-1$ supported in $\ov{D_j},\;j=1,2$.
If the corresponding total fields satisfy (\ref{eq1}) and (\ref{eq2}), then $n_1=n_2$.
\end{theorem}

To prove Theorem \ref{tt}, we need the following lemmas on the property of the total field.

\begin{lemma}\label{l1}
Let $R_2>R_1>0$ and let $D$ be a bounded domain such that $\ov{D}\subset B_{R_1}$.
Suppose $w(x,y)$ is the total field of the obstacle scattering problem (\ref{he})-(\ref{rc})
or the medium scattering problem (\ref{he-n})-(\ref{rc-n}) associated with the point source $w^i(x,y)$.
Then, for any fixed $y_0\in \pa B_{R_1}$ we have
\be\label{01}
&&w(x,y_0)\not\equiv0,\;\;\;x\in\pa B_{R_1},\;\;\;x\neq y_0,\\ \label{02}
&&w(x,y_0)\not\equiv0,\;\;\;x\in\pa B_{R_2},\\ \label{03}
&&w(x,y)\not\equiv0,\;\;\;(x,y)\in\pa B_{R_2}\times\pa B_{R_2},\;\;\;x\neq y.
\en
\end{lemma}

\begin{proof}
Since $w(x,y)$ is singular at $x=y_0$ or $y$, we know that (\ref{01}) and (\ref{03}) are true.

We now prove (\ref{02}). Assume to the contrary that $w(x,y_0)\equiv0$ for $x\in \pa B_{R_2}$,
that is, $w^s(x,y_0)=-\Phi_k(x,y_0)$ for $x\in \pa B_{R_2}$.
Then, by the uniqueness of the exterior Dirichlet problem it follows that $w^s(x,y_0)=-\Phi_k(x,y_0)$
for all $x\in\R^3\se\ov{B_{R_2}}$. Since the scattered field $w^s(x,y_0)$ is analytic for $x\in\R^3\se\ov{D}$
and $\Phi_k(x,y_0)$ is analytic for $x\in\R^3\se\{y_0\}$,
we have $w^s(x,y_0)=-\Phi_k(x,y_0)$ for all $x\in\R^3\se(\ov{D}\cup\{y_0\})$.
This is a contradiction since $\Phi_k(x,y_0)$ has a singularity at $x=y_0\in\pa B_{R_1}$ and
$w^s(x,y_0)$ is analytic when $x$ is in a neighbourhood of $y_0$. Thus, (\ref{02}) is true.
\end{proof}

\begin{lemma}\label{l2}
Under the assumption of Lemma $\ref{l1}$, we have the following results.

$(i)$ There exist two open sets $U_1,U_2\subset\pa B_{R_1}$ such that $U_1\cap U_2=\emptyset$
and $w(x,y)\neq0$ for all $(x,y)\in U_1\times U_2$.

$(ii)$ There exist two open sets $U'_1,U'_2\subset\pa B_{R_2}$ such that $U'_1\cap U'_2=\emptyset$
and $w(x,y)\neq0$ for all $(x,y)\in U'_1\times(U'_2\cup\{y_0\})$, where $y_0\in\pa B_{R_1}$.
\end{lemma}

\begin{proof}
We only prove (ii). The proof of (i) is similar.

By (\ref{02}) we know that for $y_0\in\pa B_{R_1}$ there exists $x_0\in\pa B_{R_2}$ such that $w(x_0,y_0)\neq0$.
Since $w(x,y)$ is continuous for $x,y\in\R^3\se\ov{D}$ with $x\neq y$, there exists a
neighbourhood $U'\subset \pa B_{R_2}$ of $x_0$ such that $w(x,y_0)\neq0$ for all $x\in U'$.
Further, since $w(x,y)$ is analytic with respect to $x\in \pa B_{R_2}$ and $y\in\pa B_{R_2}$,
respectively, when $x\neq y$, then it follows from (\ref{03}) that there exist two
points $x_1\in U'$ and $x_2\in \pa B_{R_2}$ such that $w(x_1,x_2)\neq0$ with $x_1\neq x_2$.
Finally, again by the continuity of $w(x,y)$ for $x,y\in\R^3\se\ov{D}$ with $x\neq y$,
there exists a neighbourhood $U'_1\subset U'$ of $x_1$ and a neighbourhood
$U'_2\subset\pa B_{R_1}$ of $x_2$ such that $U'_1\cap U'_2=\emptyset$ and $w(x,y)\neq0$
for all $(x,y)\in U'_1\times U'_2$. Thus, $w(x,y)\neq0$ for all
$(x,y)\in U'_1\times(U'_2\cup\{y_0\})$. This completes the proof.
\end{proof}

{\em Proof of Theorem $\ref{tt}$.}
From (\ref{lsp-t}) it is easy to see that (\ref{eq2}) is equivalent to the equation
\ben\label{m=}
|w_1(x,y)+w_1(x,y_0)|=|w_2(x,y)+w_2(x,y_0)|
\enn
for all $(x,y)\in(\pa B_{R_1}\times\pa B_{R_1})\cup(\pa B_{R_2}\times\pa B_{R_2})$
with $x\neq y,y_0$. This, together with (\ref{eq1}), implies that
\be\label{real}
\Rt\{w_1(x,y)\ov{w_1(x,y_0)}\}=\Rt\{w_2(x,y)\ov{w_2(x,y_0)}\}
\en
for all $(x,y)\in(\pa B_{R_1}\times\pa B_{R_1})\cup(\pa B_{R_2}\times\pa B_{R_2})$
with $x\neq y,y_0$. Define $r_j(x,y):=|w_j(x,y)|$, $j=1,2$. Then it follows from (\ref{eq1}) that
$r_1(x,y)=r_2(x,y)=:r(x,y)$, for all $x\in\pa B_{R_1},\;y\in\pa B_{R_1}\cup \pa B_{R_2}$ with $x\neq y$,
so we can write
\ben
w_j(x,y)=r(x,y)e^{i\vartheta_j(x,y)},\;\;\;\forall x,y\in\pa B_{R_1}\cup \pa B_{R_2},
\;\;x\neq y,\;\;j=1,2
\enn
with real-valued functions $\vartheta_j(x,y)$, $j=1,2$.

{\bf Case 1.} (\ref{real}) holds with $(x,y)\in\pa B_{R_1}\times\pa B_{R_1},\;x\neq y$.

Since $w_j^s(x,y),\;j=1,2,$ are analytic functions of $x\in\pa B_{R_1}$ and $y\in\pa B_{R_1}$,
respectively, and $\Phi_k(x,y)$ has a singularity at $x=y$, then, by Lemma \ref{l2}
we can choose two open sets $U_1,U_2\subset \pa B_{R_1}$ small enough so that
$U_1\cap U_2=\emptyset$, $r(x,y)\neq0$ for all $(x,y)\in U_1\times(U_2\cup{y_0})$, and
$\vartheta_j(x,y),\;j=1,2,$ are analytic with respect to $x\in U_1$ and $y\in U_2$, respectively.

Now, by (\ref{real}) we have
\be\label{cos=}
\cos[\vartheta_1(x,y)-\vartheta_1(x,y_0)]=\cos[\vartheta_2(x,y)-\vartheta_2(x,y_0)]
\en
for all $(x,y)\in U_1\times U_2$. Since $\vartheta_j(x,y),\;j=1,2,$ are real-valued
analytic functions of $x\in U_1$ and $y\in U_2$, respectively, we have either
\be\label{cos1}
\vartheta_1(x,y)-\vartheta_1(x,y_0)=\vartheta_2(x,y)-\vartheta_2(x,y_0)+2p\pi,\;\;\;
\forall (x,y)\in U_1\times U_2
\en
or
\be\label{cos2}
\vartheta_1(x,y)-\vartheta_1(x,y_0)=-[\vartheta_2(x,y)-\vartheta_2(x,y_0)]+2p\pi,\;\;\;
\forall (x,y)\in U_1\times U_2,
\en
where $p\in\Z$.

For the case when (\ref{cos1}) holds, we have
\ben
\vartheta_1(x,y)-\vartheta_2(x,y)=\vartheta_1(x,y_0)-\vartheta_2(x,y_0)
+2p\pi,\;\;\;\forall(x,y)\in U_1\times U_2.
\enn
This implies that $\alpha(x):=\vartheta_1(x,y)-\vartheta_2(x,y)
=\vartheta_1(x,y_0)-\vartheta_2(x,y_0)+2p\pi$ depends only on $x\in U_1$. Then it follows that
\ben
w_1(x,y)=r(x,y)e^{i\vartheta_1(x,y)}=r(x,y)e^{i\alpha(x)+i\vartheta_2(x,y)}=e^{i\alpha(x)}w_2(x,y)
\enn
for all $x\in U_1$ and $y\in U_2\cup\{y_0\}$. By the analyticity of $w_1(x,y)-e^{i\alpha(x)}w_2(x,y)$
in $y\in\pa B_{R_1}$ with $y\neq x$, we get
\be\label{*}
w_1(x,y)=e^{i\alpha(x)}w_2(x,y),\;\;\;\forall x\in U_1,y\in\pa B_{R_1},\;\;\;x\neq y.
\en
Changing the variables $x\rightarrow y$ and $y\rightarrow x$ in (\ref{*}) gives
\be\label{eq3}
w_1(y,x)=e^{i\alpha(y)}w_2(y,x),\;\;\;\forall x\in\pa B_{R_1},y\in U_1,\;\;\;x\neq y.
\en
Use (\ref{*}), (\ref{eq3}) and the reciprocity relation that $w^s_j(x,y)=w^s_j(y,x)$
for all $x,y\in\pa B_{R_1},\;j=1,2$ (see \cite[Theorem 3.17]{CK}) to give
\be\label{**}
e^{i\alpha(x)}w_2(x,y)=e^{i\alpha(y)}w_2(x,y),\;\;\;\forall x,y\in U_1\;\;\text{with}\;x\neq y.
\en
Since $w_j(x,y)$ has a singularity at $x=y$, and by (\ref{**}) and the analyticity of
$w_j(x,y)\;(j=1,2)$ with respect to $x\in\pa B_{R_1}$ and $y\in \pa B_{R_1}$, respectively,
with $x\neq y$, it follows that $e^{i\alpha(x)}=e^{i\alpha(y)}$ for all $x,y\in U_1$ with $x\neq y$.
This means that $e^{i\alpha(x)}\equiv e^{i\alpha}$ for all $x\in U_1$, where $\alpha$ is a real constant.
Substituting this formula into (\ref{*}) gives that 
$w_1(x,y)=e^{i\alpha}w_2(x,y)$ for all $x\in U_1,y\in\pa B_{R_1}$ with $x\neq y$.
Again, by the analyticity of $w_j(x,y)\;(j=1,2)$ with respect to $x\in\pa B_{R_1}$ with $x\neq y$ we have
\be\label{1}
w_1(x,y)=e^{i\alpha}w_2(x,y)\;\forall x,y\in \pa B_{R_1}\;\text{with}\;x\neq y,
\en
which gives
\be\label{sing}
w_1^s(x,y)-e^{i\alpha}w_2^s(x,y)=(e^{i\alpha}-1)\Phi_k(x,y),\;\;\;\forall x,y\in\pa B_{R_1}
\;\;\text{with}\;x\neq y.
\en
Since $w_j^s(x,y)$, $j=1,2,$ are analytic for $x\in G$ and $y\in G$, respectively, and $\Phi_k(x,y)$
has a singularity at $x=y$, then passing the limit $y\rightarrow x$ in (\ref{sing}) gives that
$e^{i\alpha}=1$, so
\be\label{rt1}
w_1^s(x,y)=w_2^s(x,y),\;\;\;\forall x,y\in\pa B_{R_1}.
\en

For the case when (\ref{cos2}) holds, a similar argument as above gives
\be\label{eq6}
w_1(x,y)=e^{i\beta}\ov{w_2(x,y)},\;\;\;\forall x,y\in\pa B_{R_1}\;\;\text{with}\;x\neq y
\en
for a real constant $\beta$, that is,
\ben
w_1^s(x,y)-e^{i\beta}\ov{w_2^s(x,y)}=e^{i\beta}\ov{\Phi_k(x,y)}-\Phi_k(x,y),\;\;\;
\forall x,y\in\pa B_{R_1}\;\;\text{with}\;x\neq y.
\enn
Since $w_j^s(x,y),$ $j=1,2,$ are analytic for $x\in G$ and $y\in G$, respectively,
$\Rt[\Phi_k(x,y)]$ has a singularity at $x=y$ and $\I[\Phi_k(x,y)]$ is analytic for all
$x,y\in\R^3$, then $e^{i\beta}=1$. Thus, it follows from (\ref{eq6}) that
\be\label{rt2}
w_1(x,y)=\ov{w_2(x,y)},\;\;\;\forall x,y\in\pa B_{R_1}\;\;\text{with}\;x\neq y.
\en

{\bf Case 2.} (\ref{real}) holds with $(x,y)\in\pa B_{R_2}\times \pa B_{R_2},\;x\neq y$.

By a similar argument as in Case 1, it can be obtained that there holds either
\be\label{rt3}
w_1^s(x,y)=w_2^s(x,y),\;\;\;\forall x\in\pa B_{R_2},\;y\in\pa B_{R_2}\cup\{y_0\}
\en
or
\be\label{rt4}
w_1(x,y)=\ov{w_2(x,y)},\;\;\;\forall x\in\pa B_{R_2},\;y\in\pa B_{R_2}\cup\{y_0\}\;\;\;\text{with}\;x\neq y.
\en

We now prove that both (\ref{rt2}) and (\ref{rt4}) can not hold simultaneously.
Suppose this is not the case. Then define $v(x):=w_1(x,y_0)-\ov{w_2(x,y_0)}$ for $x\in G$ with $x\neq y_0$.
Since $\Phi_k(x,y)-\ov{\Phi_k(x,y)}={i\sin(k|x-y|)}/{(2\pi|x-y|)}$ is analytic for all $x,y\in\R^3$,
then, by the analyticity of $w^s_j(x,y)\;(j=1,2)$ with respect to $x\in G$ it follows that
$v$ can be extended as an analytic function of $x\in G$, denoted by $v$ again.
Further, since ${i\sin(k|x-y|)}/{(2\pi|x-y|)}$ and $w^s_j(x,y)\;(j=1,2)$ as functions of $x$ satisfy
the Helmholtz equation $\Delta u+k^2u=0$ in $G$, we have by (\ref{rt2}) and (\ref{rt4}) that
$v$ satisfies the Dirichlet boundary value problem:
\ben
\left\{\begin{array}{ll}
\ds\Delta v+k^2v=0 & \text{in}\ B_{R_2}\ba\ov{B_{R_1}},\\
\ds v=0 & \text{on}\ \pa B_{R_1}\cup\pa B_{R_2}.
\end{array}\right.
\enn
From the assumption that $k^2$ is not a Dirichlet eigenvalue of $-\Delta$ in $B_{R_2}\ba\ov{B_{R_1}}$,
it is known that $v(x)=0$ for any $x\in B_{R_2}\ba\ov{B_{R_1}}$.
Thus $w_1(x,y_0)=\ov{w_2(x,y_0)}$ for all $x\in B_{R_2}\ba\ov{B_{R_1}}$ with $x\neq y_0$.
By the analyticity of $w_j(x,y_0)\;(j=1,2)$ with respect to $x\in G$ with $x\neq y_0$, we obtain
\be\label{conj}
w_1(x,y_0)=\ov{w_2(x,y_0)},\;\;\;\forall x\in G,\;x\neq y_0,
\en
which contradicts to the fact that $w_j(x,y_0)=\Phi_k(x,y_0)+w^s_j(x,y_0),\;j=1,2,$ satisfy the Sommerfeld
radiation condition. We then conclude that both (\ref{rt2}) and (\ref{rt4}) can not hold simultaneously.
This means that at least one of the formulas (\ref{rt2}) and (\ref{rt4}) does not hold.

If (\ref{rt2}) does not hold, then it follows that (\ref{rt1}) holds. By the reciprocity relation,
the well-posedness of the exterior Dirichlet problem and the analyticity of $w_j^s(x,y)\;(j=1,2)$ with
respect to $x\in G$ and $y\in G$, respectively, it is easily derived from (\ref{rt1}) that
\be\label{tcase1}
w_1^s(x,y)=w^s_2(x,y),\;\;\;\forall x,y\in G.
\en
Then, by \cite[Theorem 2.13]{CK} and the mixed reciprocity relation $4\pi w_j^\infty(-d,z)=u_j^s(z,d)$
for all $d\in\Sp^2$ and $z\in G,\;j=1,2$ (see \cite[Theorem 3.16]{CK}) it is obtained on applying (\ref{tcase1})
that
\be\label{far=}
u_1^\infty(\hat x,d)=u_2^\infty(\hat x,d),\;\;\;\forall\hat x,d\in\Sp^2,
\en
where $u_j^\infty$ is the far-field pattern associated with the obstacle $D_j$ (or the refractive index $n_j$)
and corresponding to the incident field $u^i(x,d)=e^{ikx\cdot d},\;j=1,2$.
Similarly, if (\ref{rt4}) does not hold, then (\ref{rt3}) holds and thus we can also show that (\ref{far=})
holds.

Finally, for the case with two impenetrable obstacles $D_1$ and $D_2$, by (\ref{far=}) and
\cite[Theorem 5.6]{CK} we have $D_1=D_2$ and $\mathscr B_1=\mathscr B_2$, while
for the case with two refractive indices $n_1$ and $n_2$, we have by (\ref{far=}) and
\cite[Theorem 6.26]{Kirsch} that $n_1=n_2$. Theorem \ref{tt} is thus proved.
\hfill $\Box$

\begin{remark}\label{re2.1} {\rm
(i) Theorem \ref{tt} (a) remains true for the two-dimensional case, and
the proof is similar.

(ii) Theorem \ref{tt} (b) also holds in two dimensions if the assumption
$n_1,n_2\in L^\infty(\R^3)$ is replaced by the condition that $n_j$ is piecewise
in $W^{1,p}(D_j)$ with $p>2,\;j=1,2$. In this case,
the proof is similar except that we need Bukhgeim's result in \cite{B} (see also
the theorem in Section 4.1 in \cite{Blasten}) instead of \cite[Theorem 6.26]{Kirsch} in the proof.

(iii) Theorem \ref{tt} (b) generalizes the uniqueness results in \cite{K14,K17,K17x,N15} substantially
in the sense that our uniqueness results only need the measurement data of the modulus of the total-field
on two spheres enclosing the inhomogeneous medium at a fixed frequency,
under no smoothness assumption on the refractive index,
instead of the measurement data in a ball for each point source in a sphere for an interval of frequencies
as used in \cite{K14,K17,K17x} or in an open domain for each point source in another open domain
at a fixed frequency as used in \cite{N15}.
}
\end{remark}

\section{Inverse electromagnetic scattering with phaseless electric total field data}\label{em}
\setcounter{equation}{0}

In this section, we extend the uniqueness results in Section \ref{acoustic} for the acoustic case
to the case of inverse electromagnetic scattering problems with phaseless electric total-field data
at a fixed frequency.
%
%
In this case, we consider the measurement of the modulus of the tangential component of the electric
total-field on two spheres enclosing the scatterers, generated by superpositions of two electric
dipoles located also on the two spheres. Denote by $E_j,E_j^s,H^s_j$ and $H_j$ the electric scattered-field,
electric total-field, magnetic scattered-field and magnetic total-field, respectively, associated with the
obstacle $D_j$ (or the refractive index $n_j$) and corresponding to the incident electric field $E^i$,
$j=1,2$. Let $B_{R_2}$ denote the ball centered at the origin with radius $R_2>R_1>0$ with $\pa B_{R_2}$
denoting the boundary of $B_{R_2}$ and let $G$ denote the unbounded component of the complement of $D_1\cup D_2$.
Denote by $N_{R_j}$ and $S_{R_j}$ the north and south poles of $\pa B_{R_j}$, respectively, $j=1,2$.
See Figure \ref{em-near} for the geometry of the electromagnetic scattering problem.
\begin{figure}[!ht]
  \centering
  \vspace{-0.2cm}
  \subfigure[]{\includegraphics[width=0.5\textwidth]{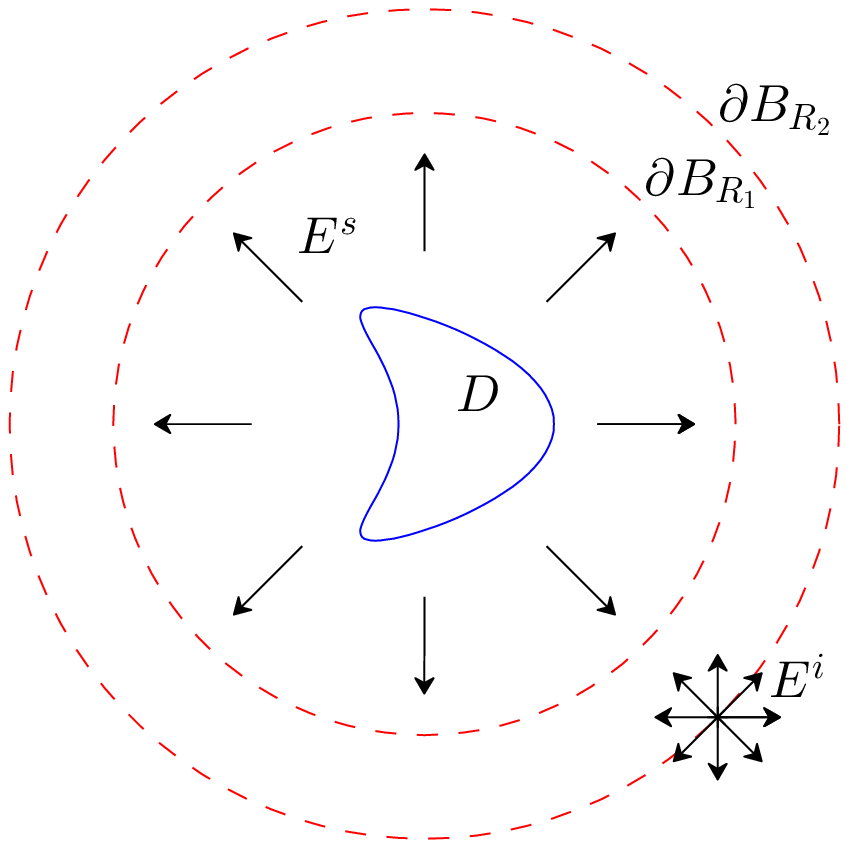}}
  \hspace{-0.5cm}
  \subfigure[]{\includegraphics[width=0.5\textwidth]{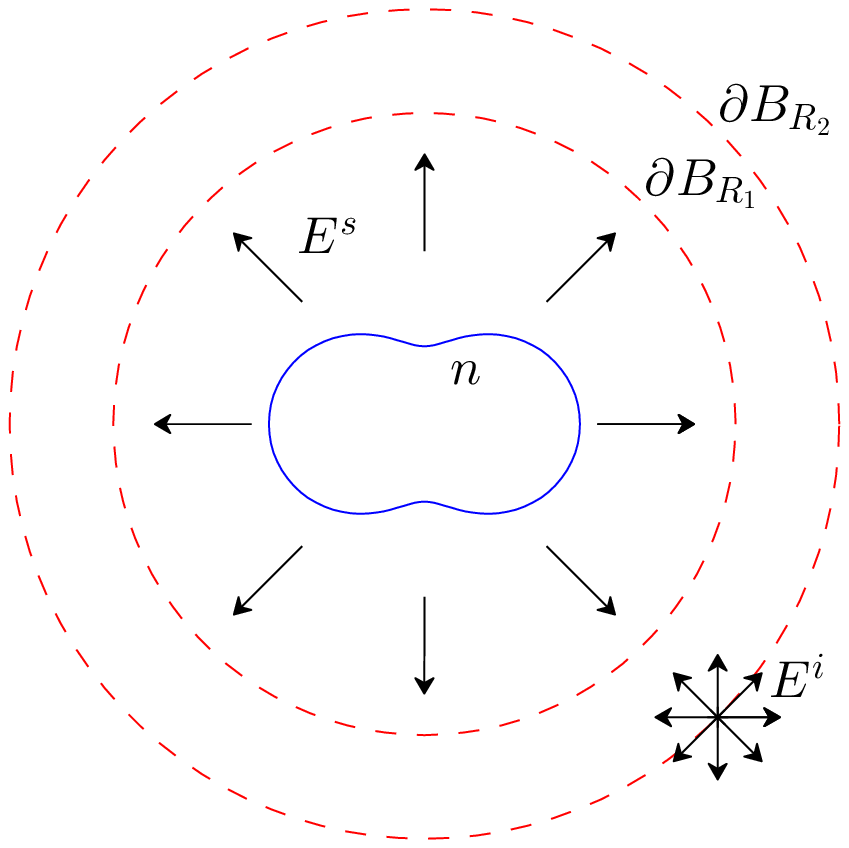}}
  \vspace{-0.3cm}
  \caption{Electromagnetic scattering by an obstacle (a) or a medium (b)}\label{em-near}
\end{figure}

By choosing appropriate $R_1$ and $R_2$ (see Lemma \ref{te}), it can be ensured that $k^2$ is not a Maxwell
eigenvalue in $B_{R_2}\se\ov{B_{R_1}}$. Here, $k^2$ is called a Maxwell eigenvalue in a bounded domain $V$
if the interior Maxwell problem
\ben
\begin{cases}
\ds {\rm curl}\,E-ikH=0 & \text{in}\;\;V\\
\ds {\rm curl}\,H+ikE=0 & \text{in}\;\;V\\
\ds \nu\times E=0 & \text{on}\;\;\pa V
\end{cases}
\enn
has a nontrivial solution $(E,H)$.

\begin{lemma}\label{te}
$k^2$ is not a Maxwell eigenvalue in $B_{R_2}\se\ov{B_{R_1}}$ if and only if
\be\no
\left|\begin{array}{cc}
j_n(kR_1) & y_n(kR_1)\\
j_n(kR_2) & y_n(kR_2)
\end{array}\right|\neq0,\;\;
\left|\begin{array}{cc}
j_n(kR_1)+kR_1j'_n(kR_1) & y_n(kR_1)+kR_1y'_n(kR_1)\\
j_n(kR_2)+kR_2j'_n(kR_2) & y_n(kR_2)+kR_2y'_n(kR_2)
\end{array}\right|\neq0\\ \label{det}
\en
for all $n=1,2,\cdots$, where $j_n$ and $y_n$ are the spherical Bessel functions and
spherical Neumann functions of order $n$, respectively.
\end{lemma}

\begin{proof}
Assume that $(E,H)$ solves the interior Maxwell problem
\be\label{iMp}
\begin{cases}
{\rm curl}E-ikH=0 & \text{in}\;\;B_{R_2}\se\ov{B_{R_1}}\\
{\rm curl}H+ikE=0 & \text{in}\;\;B_{R_2}\se\ov{B_{R_1}}\\
\nu\times E=0 & \text{on}\;\;\pa B_{R_2}\cup\pa B_{R_1}.
\end{cases}
\en
A similar argument as in the proof of \cite[Theorems 2.48 and 2.50]{KH} gives the following expansion
in the spherical vector harmonics of the electric field $E$ in $\ov{B_{R_2}}\se B_{R_1}$ as a uniformly
convergent series:
\ben
&&E(x)\\
&&=\sum_{n=1}^\infty\sum_{m=-n}^na_n^m{\rm curl}\,\left\{xj_n(k|x|)Y_n^m(\hat x)\right\}
+\sum_{n=1}^\infty\sum_{m=-n}^nb_n^m{\rm curl}\,{\rm curl}\,\left\{xj_n(k|x|)Y_n^m(\hat x)\right\}\\
&&+\sum_{n=1}^\infty\sum_{m=-n}^nc_n^m{\rm curl}\,\left\{xy_n(k|x|)Y_n^m(\hat x)\right\}+\sum_{n=1}^\infty\sum_{m=-n}^nd_n^m{\rm curl}\,{\rm curl}\,
\left\{xy_n(k|x|)Y_n^m(\hat x)\right\},\\
&&\hspace{9cm}\qquad\;x\in\ov{B_{R_2}}\se B_{R_1},
\enn
where $Y_n^m$, $m=-n,\ldots,n,$ $n = 0, 1, 2,\ldots$, are the spherical harmonics.
By \cite[(6.71) and (6.72)]{CK}, we have that for any $x\in\pa B_r$ with $r\in[R_1,R_2]$,
\ben
&&\hat x\times E(x)\\
&&=\sum_{n=1}^\infty\sum_{m=-n}^na_n^mj_n(kr){\rm Grad}Y_n^m(\hat x)
+\sum_{n=1}^\infty\sum_{m=-n}^nb_n^m\frac1{r}\left\{j_n(kr)+krj'_n(kr)\right\}
\hat x\times{\rm Grad}Y_n^m(\hat x)\\
&&+\sum_{n=1}^\infty\sum_{m=-n}^nc_n^my_n(kr){\rm Grad}Y_n^m(\hat x)
+\sum_{n=1}^\infty\sum_{m=-n}^nd_n^m\frac1{r}\left\{y_n(kr)+kry'_n(kr)\right\}
\hat x\times{\rm Grad}Y_n^m(\hat x).
\enn
By the perfectly conducting boundary condition on $\pa B_{R_2}\cup\pa B_{R_1}$ we have
\be\label{RR'=1}
&&\left(\begin{array}{cc}
j_n(kR_1) & y_n(kR_1)\\
j_n(kR_2) & y_n(kR_2)
\end{array}\right)\left(\begin{array}{c}
a_n^m\\
c_n^m
\end{array}\right)=\left(\begin{array}{c}
0\\
0
\end{array}\right),\\ \label{RR'=2}
&&\left(\begin{array}{cc}
j_n(kR_1)+kR_1j'_n(kR_1) & y_n(kR_1)+kR_1y'_n(kR_1)\\
j_n(kR_2)+kR_2j'_n(kR_2) & y_n(kR_2)+kR_2y'_n(kR_2)
\end{array}\right)\left(\begin{array}{c}
b_n^m\\
d_n^m
\end{array}\right)=\left(\begin{array}{c}
0\\
0
\end{array}\right)
\en
for all $n=1,2,\cdots$, $m=-n,\cdots,n$. By (\ref{det}) we have $a_n^m=b_n^m=c_n^m=d_n^m=0$
for all $n=1,2,\cdots$, $m=-n,\cdots,n$, and so $k^2$ is not a Maxwell eigenvalue
in $B_{R_2}\se\ov{B_{R_1}}$.

On the other hand, if
\ben
\left|\begin{array}{cc}
j_n(kR_1) & y_n(kR_1)\\
j_n(kR_2) & y_n(kR_2)
\end{array}\right|=0
\enn
or
\ben
\left|\begin{array}{cc}
j_n(kR_1)+kR_1j'_n(kR_1) & y_n(kR_1)+kR_1y'_n(kR_1)\\
j_n(kR_2)+kR_2j'_n(kR_2) & y_n(kR_2)+kR_2y'_n(kR_2)
\end{array}\right|=0
\enn
for some $n\in\N^*$, then (\ref{RR'=1}) or (\ref{RR'=2}) has non-zero solutions.
Thus there exists a nontrivial solution to the interior Maxwell problem (\ref{iMp}),
and so $k^2$ is a Maxwell eigenvalue in $B_{R_2}\se\ov{B_{R_1}}$.
The proof is thus complete.
\end{proof}

We have the following uniqueness results for the phaseless inverse electromagnetic scattering problems.

\begin{theorem}\label{ele_main}
Let $D_1,D_2$ be two bounded domains and let $R_2>R_1>0$ be large enough such that
$\ov{D_1\cup D_2}\subset B_{R_1}$ and $k^2$ is not a Maxwell eigenvalue in
$B_{R_2}\se\ov{B_{R_1}}$.

(a) Assume that $D_1$ and $D_2$ are two impenetrable obstacles with boundary
conditions $\mathscr B_1$ and $\mathscr B_2$, respectively. If the corresponding
electric total fields satisfy
\be\no
&&|{\bm e}_m(x)\cdot E_1\left(x,y_1,{\bm e}_\phi(y_1),\tau_1,y_2,{\bm e}_\theta(y_2),\tau_2\right)|\\ \label{ele_a1}
&&\qquad\qquad=|{\bm e}_m(x)\cdot E_2\left(x,y_1,{\bm e}_\phi(y_1),\tau_1,y_2,{\bm e}_\theta(y_2),\tau_2\right)|\;\;
\en
for all $x,y_1,y_2\in\pa B_{R_1}\se\{N_{R_1},S_{R_1}\}$ with $x\neq y_1,y_2$, 
$(\tau_1,\tau_2)\in\{(1,0),(0,1),(1,1)\}$, $m\in\{\phi,\theta\}$ and
\be\no
&&|{\bm e}_m(x)\cdot E_1(x,y_1,{\bm e}_n(y_1),\tau_1,y_2,{\bm e}_l(y_2),\tau_2)|\\ \label{ele_a2}
&&\qquad\qquad=|{\bm e}_m(x)\cdot E_2(x,y_1,{\bm e}_n(y_1),\tau_1,y_2,{\bm e}_l(y_2),\tau_2)|\;\;
\en
for all $x,y_1\in\pa B_{R_1}\se\{N_{R_1},S_{R_1}\}$ with $x\neq y_1$, $y_2\in\pa B_{R_2}\se\{N_{R_2},S_{R_2}\}$,
$(\tau_1,\tau_2)\in\{(0,1),(1,1)\}$, $m,n,l\in\{\phi,\theta\}$, then $D_1=D_2$
and $\mathscr B_1=\mathscr B_2$.

(b) Assume that $n_1,n_2\in C^{2,\gamma}(\R^3)$ with $\gamma>0$ are the refractive indices
of two inhomogeneous media with $n_j-1$ supported in $\ov{D_j}\;(j=1,2)$.
If the corresponding electric total fields satisfy $(\ref{ele_a1})$ and $(\ref{ele_a2})$,
then $n_1=n_2$.
\end{theorem}

\begin{remark}\label{ele_remark} {\rm
Since $E^i_j(x,y)=[E^i_j(y,x)]^\top$, and by the reciprocity relation $E^s_j(x,y)=[E^s_j(y,x)]^\top$
for all $x,y\in G$ (see \cite[Theorem 6.32]{CK}), $j=1,2$, we know that (\ref{ele_a1})
with $m=\phi$ and $(\tau_1,\tau_2)=(0,1)$ is equivalent to (\ref{ele_a1}) with $m=\theta$
and $(\tau_1,\tau_2)=(1,0)$.
}
\end{remark}

To prove Theorem \ref{ele_main}, we need some results on the phaseless electric total-fields measured
on $S_{R_1}$.

\begin{lemma}\label{ele_lem1}
Assume that the assumptions of Theorem $\ref{ele_main}$ are satisfied.
If for any fixed $m\in\{\phi,\theta\}$ there holds 
\be\no
&&|{\bm e}_m(x)\cdot E_1(x,y_1,{\bm e}_\phi(y_1),\tau_1,y_2,{\bm e}_\theta(y_2),\tau_2)|\\ \label{ele_01}
&&\qquad\qquad=|{\bm e}_m(x)\cdot E_2(x,y_1,{\bm e}_\phi(y_1),\tau_1,y_2,{\bm e}_\theta(y_2),\tau_2)|
\en
for all $x,y_1,y_2\in\pa B_{R_1}\se\{N_{R_1},S_{R_1}\}$ with $x\neq y_1,y_2$, 
$(\tau_1,\tau_2)\in\{(1,0),(0,1),(1,1)\}$, then we have either
\be\no
{\bm e}_m(x)\cdot E_1(x,y_1){\bm e}_\phi(y_1)&=&{\bm e}_m(x)\cdot E_2(x,y_1){\bm e}_\phi(y_1),\\
&&\forall x,y_1\in \pa B_{R_1}\se\{N_{R_1},S_{R_1}\},x\neq y_1,\;\;\label{ele_019}
\en
\be\no
{\bm e}_m(x)\cdot E_1(x,y_2){\bm e}_\theta(y_2)&=&{\bm e}_m(x)\cdot E_2(x,y_2){\bm e}_\theta(y_2),\\
&&\forall x,y_2\in \pa B_{R_1}\se\{N_{R_1},S_{R_1}\},x\neq y_2\;\;\label{ele_020}
\en
or
\be\no
{\bm e}_m(x)\cdot E_1(x,y_1){\bm e}_\phi(y_1)&=&-{\bm e}_m(x)\cdot\ov{E_2(x,y_1)}{\bm e}_\phi(y_1),\\
&&\forall x,y_1\in \pa B_{R_1}\se\{N_{R_1},S_{R_1}\},x\neq y_1,\;\;\label{ele_027}
\en
\be\no
{\bm e}_m(x)\cdot E_1(x,y_2){\bm e}_\theta(y_2)&=&-{\bm e}_m(x)\cdot\ov{E_2(x,y_2)}{\bm e}_\theta(y_2),\\
&&\forall x,y_2\in \pa B_{R_1}\se\{N_{R_1},S_{R_1}\},x\neq y_2.\;\;\label{ele_028}
\en
\end{lemma}

\begin{proof}
We only consider the case $m=\phi$ since the case $m=\theta$ can be proved similarly.

Using (\ref{ele_lst}) and (\ref{ele_01}) and arguing similarly as in the proof of Theorem \ref{tt} give
\be\no
&&{\rm Re}\{[{\bm e}_\phi(x)\cdot E_1(x,y_1){\bm e}_\phi(y_1)]\times
[\ov{{\bm e}_\phi(x)\cdot E_1(x,y_2){\bm e}_\theta(y_2)}]\}\\ \label{ele_011}
&&\qquad={\rm Re}\{[{\bm e}_\phi(x)\cdot E_2(x,y_1){\bm e}_\phi(y_1)]\times
[\ov{{\bm e}_\phi(x)\cdot E_2(x,y_2){\bm e}_\theta(y_2)}]\}
\en
for all $x,y_1,y_2\in\pa B_{R_1}\se\{N_{R_1},S_{R_1}\}$, $x\neq y_1,y_2$.
For $x,y\in\pa B_{R_1}\se\{N_{R_1},S_{R_1}\}$, $x\neq y$, define
\ben
r_j^{(\phi\phi)}(x,y):=|{\bm e}_\phi(x)\cdot E_j(x,y){\bm e}_\phi(y)|,\;\;\;
r_j^{(\phi\theta)}(x,y):=|{\bm e}_\phi(x)\cdot E_j(x,y){\bm e}_\theta(y)|,\;\;\;j=1,2.
\enn
It then follows from (\ref{ele_01}) with $(\tau_1,\tau_2)=(1,0)$ and $(\tau_1,\tau_2)=(0,1)$ that
\ben
r_1^{(\phi\phi)}(x,y)=r_2^{(\phi\phi)}(x,y)=:r^{(\phi\phi)}(x,y),\;\;\;
r_1^{(\phi\theta)}(x,y)=r_2^{(\phi\theta)}(x,y)=:r^{(\phi\theta)}(x,y)
\enn
for all $x,y\in\pa B_{R_1}\se\{N_{R_1},S_{R_1}\},\;x\neq y$.
Therefore we can write
\ben
&&{\bm e}_\phi(x)\cdot E_j(x,y_1){\bm e}_\phi(y_1):=r^{(\phi\phi)}(x,y_1)e^{i\vartheta_j^{(\phi\phi)}(x,y_1)},
\;\;\forall x,y_1\in \pa B_{R_1}\se\{N_{R_1},S_{R_1}\},x\neq y_1,\\
&&{\bm e}_\phi(x)\cdot E_j(x,y_1){\bm e}_\theta(y_2):=r^{(\phi\theta)}(x,y_2)e^{i\vartheta_j^{(\phi\theta)}(x,y_2)},
\;\;\forall x,y_2\in \pa B_{R_1}\se\{N_{R_1},S_{R_1}\},x\neq y_2,
\enn
where $\vartheta_j^{(\phi\phi)}$ and $\vartheta_j^{(\phi\theta)}$, $j=1,2$ are real-valued functions.

We now prove that $r^{(\phi\phi)}(x,y_1)\not\equiv0$, $x,y_1\in\pa B_{R_1}\se\{N_{R_1},S_{R_1}\}$, $x\neq y_1$,
and $r^{(\phi\theta)}(x,y_2)\not\equiv0$, $x,y_2\in\pa B_{R_1}\se\{N_{R_1},S_{R_1}\}$, $x\neq y_2$.
In fact, fix $y_1\in\pa B_{R_1}\se\{N_{R_1},S_{R_1}\}$ and define the circle
$C_{\bm e_\phi(y_1)}:=\{x\in\pa B_{R_1}:(x-y_1)\cdot\bm e_\phi(y_1)=0\}$,
which is the intersection of the sphere $\pa B_{R_1}$ with the plane whose normal vector
is $\bm e_\phi(y_1)$ at $y_1$. When $x$ tend to $y_1$ along the circle $C_{\bm e_\phi(y_1)}$,
we have $\widehat{x-y_1}^\top\bm e_\phi(y_1)\to0$ and ${\bm e}_\phi(x)\cdot{\bm e}_\phi(y_1)\to1$.
Thus, by (\ref{ele_ei}) it is known that
\be\label{ele_sing}
{\bm e}_\phi(x)\cdot E^i(x,y_1){\bm e}_\phi(y_1)
\sim\frac ik\left[k^2+\left(ik-\frac1{|x-y_1|}\right)\frac1{|x-y_1|}\right]\Phi_k(x,y_1)
\en
as $x$ goes to $y_1$ along the circle $C_{\bm e_\phi(y_1)}$.
The singularity in (\ref{ele_sing}) implies that $r^{(\phi\phi)}(x,y_1)\not\equiv0$ for
$x,y_1\in\pa B_{R_1}\se\{N_{R_1},S_{R_1}\}$ with $x\neq y_1$ since $E_j^s(x,y_1)$ is analytic
with respect to $x,y_1\in\pa B_{R_1}\se\{N_{R_1},S_{R_1}\}$ with $x\neq y_1$, respectively ($j=1,2$).
Further, fix $y_2\in\pa B_{R_1}\se\{N_{R_1},S_{R_1}\}$ and define the circle
\ben
C_{\bm e_\phi(y_2)+\bm e_\theta(y_2)}:=\{x\in\pa B_{R_1}:(x-y_2)\cdot(\bm e_\phi(y_2)
+\bm e_\theta(y_2))=0\}.
\enn
Then, on letting $x$ tend to $y_2$ along $C_{\bm e_\phi(y_2)+\bm e_\theta(y_2)}$
we have $\bm e_\phi(x)\cdot\bm e_\theta(y_2)\to0$ and
${\bm e}_\phi(x)\cdot\left[\widehat{x-y_2}\cdot\widehat{x-y_2}^\top{\bm e}_\theta(y_2)\right]\to c_1$
for a non-zero constant $c_1$. Thus it follows from (\ref{ele_ei}) that
\be\label{ele_sing'}
{\bm e}_\phi(x)\cdot E^i(x,y_2){\bm e}_\theta(y_2)=\frac{1}{|x-y_2|^2}\Phi_k(x,y_2)\left[c_2+o(1)\right]
\en
as $x\to y_2$ along $C_{\bm e_\phi(y_2)+\bm e_\theta(y_2)}$, where $c_2$ is a non-zero constant.
Therefore the singularity in (\ref{ele_sing'}) implies that $r^{(\phi\theta)}(x,y_2)\not\equiv0$
for $x,y_2\in\pa B_{R_1}\se\{N_{R_1},S_{R_1}\}$ with $x\neq y_2$ since $E_j^s(x,y_2)$ is analytic
with respect to $x,y_2\in\pa B_{R_1}\se\{N_{R_1},S_{R_1}\}$, $x\neq y_2$, respectively ($j=1,2$).
Then, similarly as in the proof of Theorem \ref{tt}, we can show that there are three small enough open sets
$U,U_1,U_2\subset\pa B_{R_1}\se\{N_{R_1},S_{R_1}\}$ such that $U,U_1$ and $U_2$ are disjoint,
$r^{(\phi\phi)}(x,y_1)\neq0$ and $r^{(\phi\theta)}(x,y_2)\neq0$ for all $x\in U$, $y_1\in U_1$
and $y_2\in U_2$, and $\vartheta_j^{(\phi\phi)}(x,y_1)$ and $\vartheta_j^{(\phi\theta)}(x,y_2)$
are analytic with respect to $x\in U$, $y_1\in U_1$, $y_2\in U_2$, respectively, $j=1,2$.

Now, by (\ref{ele_011}) we have
\be\label{ele_012}
\cos[\vartheta_1^{(\phi\phi)}(x,y_1)-\vartheta_1^{(\phi\theta)}(x,y_2)]
=\cos[\vartheta_2^{(\phi\phi)}(x,y_1)-\vartheta_2^{(\phi\theta)}(x,y_2)]
\en
for all $(x,y_1,y_2)\in U\times U_1\times U_2$. Since $\vartheta_j^{(\phi\phi)}(x,y_1)$
and $\vartheta_j^{(\phi\theta)}(x,y_2)$ are analytic functions of $x\in U$, $y_1\in U_1$
and $y_2\in U_2$, respectively ($j=1,2$), we obtain that there holds either
\be\label{ele_013}
\vartheta_1^{(\phi\phi)}(x,y_1)-\vartheta_1^{(\phi\theta)}(x,y_2)
=\vartheta_2^{(\phi\phi)}(x,y_1)-\vartheta_2^{(\phi\theta)}(x,y_2)+2p\pi
\en
or
\be\label{ele_014}
\vartheta_1^{(\phi\phi)}(x,y_1)-\vartheta_1^{(\phi\theta)}(x,y_2)
=-[\vartheta_2^{(\phi\phi)}(x,y_1)-\vartheta_2^{(\phi\theta)}(x,y_2)]+2p\pi
\en
for all $(x,y_1,y_2)\in U\times U_1\times U_2$, where $p\in\Z$.

For the case when (\ref{ele_013}) holds, we have
\ben
\alpha(x):=\vartheta_1^{(\phi\phi)}(x,y_1)-\vartheta_2^{(\phi\phi)}(x,y_1)
=\vartheta_1^{(\phi\theta)}(x,y_2)-\vartheta_2^{(\phi\theta)}(x,y_2)+2p\pi
\enn
depends only on $x$, which is a real-valued analytic function in $x\in U$. Thus
\ben
{\bm e}_\phi(x)\cdot E_1(x,y_1){\bm e}_\phi(y_1)&=&r^{(\phi\phi)}(x,y_1)e^{i\vartheta_1^{(\phi\phi)}(x,y_1)}\\
&=&r^{(\phi\phi)}(x,y_1)e^{i\alpha(x)+i\vartheta_2^{(\phi\phi)}(x,y_1)}\\
&=&e^{i\alpha(x)}{\bm e}_\phi(x)\cdot E_2(x,y_1){\bm e}_\phi(y_1),\\
{\bm e}_\phi(x)\cdot E_1(x,y_2){\bm e}_\theta(y_2)
&=&r^{(\phi\theta)}(x,y_2)e^{i\vartheta_1^{(\phi\theta)}(x,y_2)}\\
&=&r^{(\phi\theta)}(x,y_2)e^{i\alpha(x)+i\vartheta_2^{(\phi\theta)}(x,y_2)}\\
&=&e^{i\alpha(x)}{\bm e}_\phi(x)\cdot E_2(x,y_2){\bm e}_\theta(y_2)
\enn
for all $(x,y_1,y_2)\in U\times U_1\times U_2$. By the analyticity of
$E_1(x,y)-e^{i\alpha(x)}E_2(x,y)$ in $y\in \pa B_{R_1}$ for $y\neq x$, we obtain
\be\no
&&{\bm e}_\phi(x)\cdot E_1(x,y_1){\bm e}_\phi(y_1)\\ \label{ele_015}
&&\qquad\;=e^{i\alpha(x)}{\bm e}_\phi\cdot E_2(x,y_1){\bm e}_\phi(y_1),\;\;
\forall x\in U,y_1\in \pa B_{R_1}\se\{N_{R_1},S_{R_1}\},x\neq y_1,\qquad\;\;\\ \no
&&{\bm e}_\phi(x)\cdot E_1(x,y_2){\bm e}_\theta(y_2)\\ \label{ele_016}
&&\qquad\;=e^{i\alpha(x)}{\bm e}_\phi\cdot E_2(x,y_2){\bm e}_\theta(y_2),\;\;
\forall x\in U,y_2\in \pa B_{R_1}\se\{N_{R_1},S_{R_1}\},x\neq y_2.\qquad\;\;
\en
From (\ref{ele_015}) it follows that
\be\no
{\bm e}_\phi(x)\cdot[E^s_1(x,y_1){\bm e}_\phi(y_1)-e^{i\alpha(x)}E^s_2(x,y_1){\bm e}_\phi(y_1)]
=[e^{i\alpha(x)}-1]{\bm e}_\phi(x)\cdot E^i(x,y_1){\bm e}_\phi(y_1)\\ \label{ele_1sing}
\en
for all $x\in U$ and $y_1\in\pa B_{R_1}\se\{N_{R_1},S_{R_1}\}$ with $x\neq y_1$.
For arbitrarily fixed $y_1\in U$, the left-hand side of (\ref{ele_1sing}) is analytic in $x\in U$,
while, by (\ref{ele_sing}) the right-hand side of (\ref{ele_1sing}) is singular when $x$ is close
to $y_1$ along the circle $C_{\bm e_\phi(y_1)}$. Therefore, $e^{i\alpha(y_1)}=1$. Since $y_1\in U$ is
arbitrary, we have $e^{i\alpha(x)}=1$ for all $x\in U$, and so (\ref{ele_015}) and (\ref{ele_016}) become
\be\no
&&{\bm e}_\phi(x)\cdot E_1(x,y_1){\bm e}_\phi(y_1)\\ \label{ele_017}
&&\qquad\qquad={\bm e}_\phi\cdot E_2(x,y_1){\bm e}_\phi(y_1),\;\;
\forall x\in U,y_1\in \pa B_{R_1}\se\{N_{R_1},S_{R_1}\},x\neq y_1,\qquad\;\;\\ \no
&&{\bm e}_\phi(x)\cdot E_1(x,y_2){\bm e}_\theta(y_2)\\ \label{ele_018}
&&\qquad\qquad={\bm e}_\phi\cdot E_2(x,y_2){\bm e}_\theta(y_2),\;\;
\forall x\in U,y_2\in \pa B_{R_1}\se\{N_{R_1},S_{R_1}\},x\neq y_2.\qquad\;\;
\en
This, together with the analyticity of $E_j(x,y)\;(j=1,2)$ in $x\in \pa B_{R_1}$ with $x\neq y$,
gives (\ref{ele_019}) and (\ref{ele_020}).

Similarly, for the case when (\ref{ele_014}) holds, we can deduce
\be\no
&&{\bm e}_\phi(x)\cdot E_1(x,y_1){\bm e}_\phi(y_1)\\ \label{ele_022}
&&\qquad\;=e^{i\beta(x)}{\bm e}_\phi(x)\cdot\ov{E_2(x,y_1)}{\bm e}_\phi(y_1),\;\;
\forall x\in U,y_1\in \pa B_{R_1}\se\{N_{R_1},S_{R_1}\},x\neq y_1,\qquad\;\;\\ \no
&&{\bm e}_\phi(x)\cdot E_1(x,y_2){\bm e}_\theta(y_2)\\ \label{ele_023}
&&\qquad\;=e^{i\beta(x)}{\bm e}_\phi(x)\cdot\ov{E_2(x,y_2)}{\bm e}_\theta(y_2),\;\;
\forall x\in U,y_2\in \pa B_{R_1}\se\{N_{R_1},S_{R_1}\},x\neq y_2,\qquad\;\;
\en
where $\beta$ is a real-valued analytic function of $x\in U$.
From (\ref{ele_022}) it is easy to derive that
\be\no
&&{\bm e}_\phi(x)\cdot[E_1^s(x,y_1)-e^{i\beta(x)}\ov{E_2^s(x,y_1)}]{\bm e}_\phi(y_1)\\ \label{ele_2sing}
&&\qquad\qquad={\bm e}_\phi(x)\cdot[e^{i\beta(x)}\ov{E^i(x,y_1)}-E^i(x,y_1)]{\bm e}_\phi(y_1)\;\;
\en
for all $x\in U$, $y_1\in \pa B_{R_1}$, $x\neq y_1$.
For arbitrarily fixed $y_1\in U$, the left-hand side of (\ref{ele_2sing}) is analytic in $x\in U$,
but, by (\ref{ele_ei}) and a direct calculation, the right-hand side of (\ref{ele_2sing}) has a singularity
at $x=y_1$ unless $e^{i\beta(x)}=-1$ for $x\in C_{\bm e_\phi(y_1)}$ near $y_1$.
This means that $e^{i\beta(y_1)}=-1$. By the arbitrariness of $y_1\in U$, we have $e^{i\beta(x)}=-1$
for all $x\in U$, and so
\be\no
e^{i\beta(x)}\ov{E^i(x,y)}-E^i(x,y)&=&-\ov{E^i(x,y)}-E^i(x,y)\\ \label{ele_sinc}
&=&(k^2I+\nabla_x\nabla_x)\frac{i}{k}\left[\ov{\Phi_k(x,y)}-\Phi_k(x,y)\right]\;\;
\en
is analytic in $x\in\R^3$ and $y\in\R^3$, respectively, since $\ov{\Phi_k(x,y)}-\Phi_k(x,y)$
is analytic in $x\in\R^3$ and $y\in\R^3$, respectively.
Thus (\ref{ele_022}) and (\ref{ele_023}) are reduced to
\be\no
&&{\bm e}_\phi(x)\cdot E_1(x,y_1){\bm e}_\phi(y_1)\\ \label{ele_025}
&&\qquad\;\;=-{\bm e}_\phi(x)\cdot\ov{E_2(x,y_1)}{\bm e}_\phi(y_1),\;\;
\forall x\in U,y_1\in \pa B_{R_1}\se\{N_{R_1},S_{R_1}\},x\neq y_1,\qquad\;\;\\ \no
&&{\bm e}_\phi(x)\cdot E_1(x,y_2){\bm e}_\theta(y_2)\\ \label{ele_026}
&&\qquad\;\;=-{\bm e}_\phi(x)\cdot\ov{E_2(x,y_2)}{\bm e}_\theta(y_2),\;\;
\forall x\in U,y_2\in \pa B_{R_1}\se\{N_{R_1},S_{R_1}\},x\neq y_2.\qquad\;\;
\en
Both (\ref{ele_027}) and (\ref{ele_028}) then follow from the analyticity of $E_j(x,y)\;(j=1,2)$
in $x\in \pa B_{R_1}$ for $x\neq y$. The proof is thus complete.
\end{proof}

\begin{lemma}\label{ele_lem}
Assume that the assumptions of Theorem $\ref{ele_main}$ are satisfied.
If for any fixed $m,n,l\in\{\phi,\theta\}$ there holds
\be\no
&&|{\bm e}_m(x)\cdot E_1(x,y_1,{\bm e}_n(y_1),\tau_1,y_2,{\bm e}_l(y_2),\tau_2)|\\ \label{ele_3}
&&\qquad\qquad=|{\bm e}_m(x)\cdot E_2(x,y_1,{\bm e}_n(y_1),\tau_1,y_2,{\bm e}_l(y_2),\tau_2)|\;\;
\en
for all $x,y_1\in\pa B_{R_1}\se\{N_{R_1},S_{R_1}\}$ with $x\neq y_1$, 
$y_2\in\pa B_{R_2}\se\{N_{R_2},S_{R_2}\}$, $(\tau_1,\tau_2)\in\{(1,0),(0,1),(1,1)\}$, 
then we have either
\be\no
{\bm e}_m(x)\cdot E_1(x,y_1){\bm e}_n(y_1)&=&{\bm e}_m(x)\cdot E_2(x,y_1){\bm e}_n(y_1),\\ \label{ele_4}
&&\;\forall x,y_1\in\pa B_{R_1}\se\{N_{R_1},S_{R_1}\}\;\textrm{with}\;x\neq y_1,\qquad\\ \no
{\bm e}_m(x)\cdot E_1(x,y_2){\bm e}_l(y_2)&=&{\bm e}_m(x)\cdot E_2(x,y_2){\bm e}_l(y_2),\\ \label{ele_5}
&&\;\forall x\in \pa B_{R_1}\se\{N_{R_1},S_{R_1}\},y_2\in \pa B_{R_2}\se\{N_{R_2},S_{R_2}\}\qquad
\en
or
\be\no
{\bm e}_m(x)\cdot E_1(x,y_1){\bm e}_n(y_1)&=&-{\bm e}_m(x)\cdot\ov{E_2(x,y_1)}{\bm e}_n(y_1),\\ \label{ele_6}
&&\;\forall x,y_1\in\pa B_{R_1}\se\{N_{R_1},S_{R_1}\}\;\textrm{with}\;x\neq y_1,\qquad\\ \no
{\bm e}_m(x)\cdot E_1(x,y_2){\bm e}_l(y_2)&=&-{\bm e}_m(x)\cdot\ov{E_2(x,y_2)}{\bm e}_l(y_2),\\ \label{ele_7}
&&\;\forall x\in\pa B_{R_1}\se\{N_{R_1},S_{R_1}\},y_2\in\pa B_{R_2}\se\{N_{R_2},S_{R_2}\}.\qquad
\en
\end{lemma}

\begin{proof}
Since $|{\bm e}_m(x)\cdot E_1(x,y_2){\bm e}_l(y_2)|$ is analytic in $x\in\left(\pa B_{R_1}\se\{N_{R_1},S_{R_1}\}\right)$ and $y_2\in\left(\pa B_{R_2}\se\{N_{R_2},S_{R_2}\}\right)$, respectively, we only need to distinguish
between two cases:
\ben\label{ele_not0}
&&\mbox{\bf A)}\;\;|{\bm e}_m(x)\cdot E_1(x,y_2){\bm e}_l(y_2)|\not\equiv0,\;\;\forall
(x,y_2)\in\left(\pa B_{R_1}\se\{N_{R_1},S_{R_1}\}\right)\times\left(\pa B_{R_2}\se\{N_{R_2},S_{R_2}\}\right),\\
&&\mbox{\bf B)}\;\;|{\bm e}_m(x)\cdot E_1(x,y_2){\bm e}_l(y_2)|\equiv0,\;\;\forall
(x,y_2)\in\left(\pa B_{R_1}\se\{N_{R_1},S_{R_1}\}\right)\times\left(\pa B_{R_2}\se\{N_{R_2},S_{R_2}\}\right).
\enn

For the case when {\bf A}) holds, by arguing similarly as in the proof of Lemma \ref{ele_lem1} it can be deduced
from (\ref{ele_3}) that we have either
\be\no
{\bm e}_m(x)\cdot E_1(x,y_1){\bm e}_n(y_1)&=&e^{i\alpha(x)}{\bm e}_m(x)\cdot E_2(x,y_1){\bm e}_n(y_1),\\ \label{le1}
&&\;\forall x\in U,\;y_1\in\pa B_{R_1}\se\{N_{R_1},S_{R_1}\},\;x\neq y_1,\\ \no
{\bm e}_m(x)\cdot E_1(x,y_2){\bm e}_l(y_2)&=&e^{i\alpha(x)}{\bm e}_m(x)\cdot E_2(x,y_2){\bm e}_l(y_2),\\ \label{le2}
&&\;\forall x\in U,\;y_2\in\pa B_{R_2}\se\{N_{R_2},S_{R_2}\}
\en
or
\be\no
{\bm e}_m(x)\cdot E_1(x,y_1){\bm e}_n(y_1)&=&e^{i\beta(x)}{\bm e}_m(x)\cdot\ov{E_2(x,y_1)}{\bm e}_n(y_1),\\ \label{le3}
&&\;\forall x\in U,\;y_1\in\pa B_{R_1}\se\{N_{R_1},S_{R_1}\},\;x\neq y_1,\\ \no
{\bm e}_m(x)\cdot E_1(x,y_2){\bm e}_l(y_2)&=&e^{i\beta(x)}{\bm e}_m(x)\cdot\ov{E_2(x,y_2)}{\bm e}_l(y_2),\\ \label{le4}
&&\;\forall x\in U,\;y_2\in\pa B_{R_2}\se\{N_{R_2},S_{R_2}\},
\en
where $U$ is some small open subset of $\pa B_{R_1}\se\{N_{R_1},S_{R_1}\}$,
and $\alpha(x)$ and $\beta(x)$ are real-valued functions of $x$.
By (\ref{ele_019}) and (\ref{ele_027}) in Lemma \ref{ele_lem1} it follows easily that
$e^{i\alpha(x)}=1$ and $e^{i\beta(x)}=-1$. This, together with (\ref{le1})-(\ref{le4}) and the analyticity
of the total fields $E_j(x,y)$, $j=1,2$, in $x$ for $x\not=y$, implies that either (\ref{ele_4}) and (\ref{ele_5})
hold or (\ref{ele_6}) and (\ref{ele_7}) hold.

For the case when {\bf B}) holds, it follows from (\ref{ele_3}) that
\ben
|{\bm e}_m(x)\cdot E_2(x,y_2){\bm e}_l(y_2)|\equiv0,\;\;\forall
(x,y_2)\in\left(\pa B_{R_1}\se\{N_{R_1},S_{R_1}\}\right)\times\left(\pa B_{R_2}\se\{N_{R_2},S_{R_2}\}\right).
\enn
Therefore, both (\ref{ele_5}) and (\ref{ele_7}) hold. Further, by Lemma \ref{ele_lem1}
we have that either (\ref{ele_4}) or (\ref{ele_6}) holds. The proof is thus complete.
\end{proof}

Using Lemmas \ref{ele_lem1} and \ref{ele_lem} we can prove the following lemma.

\begin{lemma}\label{ele_lemma}
Assume that the assumptions of Theorem $\ref{ele_main}$ are satisfied.
If $(\ref{ele_01})$ and $(\ref{ele_3})$ hold for all $m,n,l\in\{\phi,\theta\}$, then we have
\be\label{ele_12}
E_1(x,y)=E_2(x,y),\;\;\;\forall x,y\in G,\;x\neq y.
\en
\end{lemma}

\begin{proof}
We first show that for any fixed $m\in\{\phi,\theta\}$,
\be\no
{\bm e}_m(x)\cdot E_1(x,y_1){\bm e}_n(y_1)&=&{\bm e}_m(x)\cdot E_2(x,y_1){\bm e}_n(y_1),\\ \label{ele_8}
&&\;\forall x,y_1\in\pa B_{R_1}\se\{N_{R_1},S_{R_1}\},\;x\neq y_1,\;
\forall n\in\{\phi,\theta\}.\quad\quad
\en
To this end, for any fixed $m\in\{\phi,\theta\}$ we need to distinguish
between the following two cases.

\tb{Case 1.} $\Rt[{\bm e}_m(x)\cdot E_1(x,y_1){\bm e}_l(y_1)]=0$ for all
$x,y_1\in\pa B_{R_1}\se\{N_{R_1},S_{R_1}\}$ with $x\neq y_1$ and for all $l\in\{\phi,\theta\}$.

In this case, by Lemma \ref{ele_lem1} it follows that $\Rt[{\bm e}_m(x)\cdot E_2(x,y_1){\bm e}_l(y_1)]=0$
for all $x,y_1\in\pa B_{R_1}\se\{N_{R_1},S_{R_1}\}$ with $x\neq y_1$ and for all $l\in\{\phi,\theta\}$.
By Lemma \ref{ele_lem1} again we have (\ref{ele_8}).

\tb{Case 2.} $\Rt[{\bm e}_m(x)\cdot E_1(x,y_1){\bm e}_l(y_1)]\not=0$ for some
$x,y_1\in\pa B_{R_1}\se\{N_{R_1},S_{R_1}\}$ with $x\neq y_1$, $l\in\{\phi,\theta\}$.
Here, we only consider the case with $l=\phi$. The case $l=\theta$ can be treated similarly.

In this case, by Lemma \ref{ele_lem1} we have that either both (\ref{ele_019}) and (\ref{ele_020}) hold
or both (\ref{ele_027}) and (\ref{ele_028}) hold.
We can prove that both (\ref{ele_027}) and (\ref{ele_028}) can not hold simultaneously.
Suppose this is not the case. Then we have
\be\no
{\bm e}_m(x)\cdot [E_1(x,y_1){\bm e}_n(y_1)]&=&-{\bm e}_m(x)\cdot [E_2(x,y_1){\bm e}_n(y_1)],\\ \label{ele_10}
&&\;\forall x,y_1\in\pa B_{R_1}\se\{N_{R_1},S_{R_1}\},\;x\neq y_1,\;
\forall n\in\{\phi,\theta\}.\;\;\quad
\en
This, together with Lemmas \ref{ele_lem1} and \ref{ele_lem}, implies that
\be\label{ele_11}
{\bm e}_m(x)\cdot [E_1(x,y_2){\bm e}_n(y_2)]&=&-{\bm e}_m(x)\cdot[\ov{E_2(x,y_2)}{\bm e}_n(y_2)],
\;\;\forall n\in\{\phi,\theta\}\\ \no
&&\forall x\in\pa B_{R_1}\se\{N_{R_1},S_{R_1}\},y_2\in\pa B_{R_2}\se\{N_{R_2},S_{R_2}\}.
\en
We now show that both (\ref{ele_10}) and (\ref{ele_11}) can not hold simultaneously.
In fact, by the reciprocity relation $E_j(x,y)=[E_j(y,x)]^\top$ for all $x,y\in G$ ($j=1,2$),
we deduce from (\ref{ele_10}) and (\ref{ele_11}) that
\be\label{ele_13}
{\bm e}_n(y_1)\cdot [E_1(y_1,x){\bm e}_m(x)]&=&-{\bm e}_n(y_1)\cdot[\ov{E_2(y_1,x)}{\bm e}_m(x)],
\;\;\forall n\in\{\phi,\theta\},\\ \no
&&\forall x,y_1\in\pa B_{R_1}\se\{N_{R_1},S_{R_1}\},x\neq y_1,\\ \label{ele_14}
{\bm e}_n(y_2)\cdot [E_1(y_2,x){\bm e}_m(x)]&=&-{\bm e}_n(y_2)\cdot[\ov{E_2(y_2,x)}{\bm e}_m(x)],
\;\;\forall n\in\{\phi,\theta\}\\ \no
&&\forall x\in\pa B_{R_1}\se\{N_{R_1},S_{R_1}\},y_2\in \pa B_{R_2}\se\{N_{R_2},S_{R_2}\}.
\en
This, together with the linear combination of ${\bm e}_\phi(y_j)$ and ${\bm e}_\theta(y_j)$ ($j=1,2$),
gives that
\be\label{ele_15}
\nu(y_1)\times [E_1(y_1,x){\bm e}_m(x)]&=&-\nu(y_1)\times[\ov{E_2(y_1,x)}{\bm e}_m(x)],\\ \no
&&\forall x,y_1\in \pa B_{R_1}\se\{N_{R_1},S_{R_1}\},x\neq y_1,\\ \label{ele_16}
\nu(y_2)\times [E_1(y_2,x){\bm e}_m(x)]&=&-\nu(y_2)\times[\ov{E_2(y_2,x)}{\bm e}_m(x)],\\ \no
&&\forall x\in \pa B_{R_1}\se\{N_{R_1},S_{R_1}\},y_2\in \pa B_{R_2}\se\{N_{R_2},S_{R_2}\}.
\en
For any fixed $x\in \pa B_{R_1}\se\{N_{R_1},S_{R_1}\}$ and $m\in\{\phi,\theta\}$,
define $\wid{E}(y):=E_1(y,x){\bm e}_m(x)+\ov{E_2(y,x)}{\bm e}_m(x)$, $y\neq x$.
Since $2\Rt[E^i(y,x)]:=E^i(y,x)+\ov{E^i(y,x)}$ is analyticity for all $x,y\in\R^3$ (see (\ref{ele_sinc})),
then, by the analyticity of $E_j^s(y,x)$ with respect to $y\in G$ ($j=1,2$), it follows
that $\wid{E}$ can be extended as an analytic function of $y\in G$,
which we denote by $\wid{E}$ again. Define $\wid{H}(y):=[1/(ik)]{\rm curl}_y\wid{E}(y)$.
Then $(\Rt[E^i(y,x)]\bm e_m(x),\I[H^i(y,x)]\bm e_m(x))$ and $(E_j^s(y,x)\bm e_m(x),H_j^s(y,x)\bm e_m(x))$
satisfy the Maxwell equations for $x\in G$, $j=1,2$.
Thus it follows by (\ref{ele_15}), (\ref{ele_16}) and the analyticity of $E_j(y,x)$ in $y\in G$
with $y\neq x$ ($j=1,2$) that $(\wid{E},\wid{H})$ satisfies the interior Maxwell problem
\ben
\begin{cases}
{\rm curl}\,\wid{E}-ik\wid{H}=0 & \text{in}\;\;B_{R_2}\se\ov{B_{R_1}},\\
{\rm curl}\,\wid{H}+ik\wid{E}=0 & \text{in}\;\;B_{R_2}\se\ov{B_{R_1}},\\
\nu\times\wid{E}=0 & \text{on}\;\;\pa B_{R_1}\cup\pa B_{R_2}.
\end{cases}
\enn
Since $k^2$ is not a Maxwell eigenvalue in $B_{R_2}\se\ov{B_{R_1}}$, then $\wid{E}=0$
in $B_{R_2}\se\ov{B_{R_1}}$. Thus, and by the analyticity of $E_j(y,x)$ in $y\in G$ with $y\neq x$
($j=1,2$), we have $E_1(y,x){\bm e}_\phi(x)=-\ov{E_2(y,x)}{\bm e}_\phi(x)$ for all $y\in G$, $y\not=x$.
This contradicts to the fact that $E_j(y,x)\bm e_m(x)=E^i(y,x)\bm e_m(x)+E_j^s(y,x)\bm e_m(x)$,
$j=1,2$, satisfy the Silver-M\"uller radiation condition. Therefore, (\ref{ele_10}) and
(\ref{ele_11}) can not be true simultaneously, which means that both (\ref{ele_027}) and
(\ref{ele_028}) can not hold simultaneously.
This then implies that both (\ref{ele_019}) and (\ref{ele_020}) are true, and so (\ref{ele_8}) holds.

Finally, by (\ref{ele_8}) and the linear combination of ${\bm e}_\phi$ and ${\bm e}_\theta$
we obtain that for arbitrarily fixed $y_1\in\pa B_{R_1}\se\{N_{R_1},S_{R_1}\}$ and $n\in\{\phi,\theta\}$,
\ben\label{ele-a}
\nu(x)\times [E^s_1(x,y_1){\bm e}_n(y_1)]=\nu(x)\times [E^s_2(x,y_1){\bm e}_n(y_1)],\;\;
\forall x\in\pa B_{R_1}\se\{N_{R_1},S_{R_1}\}.
\enn
By the well-posedness of the exterior Maxwell problem in $\R^3\se B_{R_1}$
with the PEC condition on $\pa B_{R_1}$ it is deduced that
for arbitrarily fixed $y_1\in\pa B_{R_1}\se\{N_{R_1},S_{R_1}\}$,
\ben
E^s_1(x,y_1){\bm e}_n(y_1)=E^s_2(x,y_1){\bm e}_n(y_1),\;\;\;\forall\; x\in\R^3\se B_{R_1},
\;\;\forall\; n\in\{\phi,\theta\}.
\enn
This, together with the reciprocity relation $E^s_j(x,y)=[E^s_j(y,x)]^\top$ for all $x,y\in G$, $j=1,2$,
implies that for any fixed $x\in\R^3\se B_{R_1}$,
\ben
\nu(y)\times E^s_1(y,x)=\nu(y)\times E^s_2(y,x),\;\;\forall y\in\pa B_{R_1}.
\enn
Again, by the well-posedness of the exterior Maxwell problem in $\R^3\se B_{R_1}$
with the PEC condition on $\pa B_{R_1}$ it is derived that for any fixed $x\in\R^3\se B_{R_1}$,
\ben
E^s_1(y,x)=E^s_2(y,x),\;\;\forall y\in\R^3\se B_{R_1}.
\enn
The required result (\ref{ele_12}) then follows from this, the reciprocity relation
and the analyticity of $E^s_j(x,y)\;(j=1,2)$ in $x\in G$ and $y\in G$, respectively.
\end{proof}

{\em Proof of Theorem $\ref{ele_main}$.}
By Lemma \ref{ele_lemma} it follows from (\ref{ele_a1}) and (\ref{ele_a2}) that (\ref{ele_12}) holds.
For $j=1,2$, denote by $E_j^\infty(\hat{x},y)$ the far-field pattern of $E_j^s(x,y)$, $x,y\in G$,
and by $E_j^s(x,d)$ and $E_j^\infty(\hat{x},d)$ the electric scattered field and its
far-field pattern associated with the obstacle $D_j$ (or the refractive index $n_j$) and
corresponding to the incident electromagnetic plane waves
described by the matrices $E^i(x,d)$, $H^i(x,d)$ defined by
\ben
&&E^i(x,d)p:=\frac ik{\rm curl\;curl}\,pe^{ikx\cdot d}=ik(d\times p)\times de^{ikx\cdot d},\\
&&H^i(x,d)p:={\rm curl}\,pe^{ikx\cdot d}=ikd\times pe^{ikx\cdot d},
\enn
where $d\in\Sp^2$ and $p\in\R^3$ denote the incident direction and polarization vector, respectively,
and $x\in\R^3$. Then, by (\ref{ele_12}) in Lemma \ref{ele_lemma} and
the mixed reciprocity relation that $4\pi E_j^\infty(-d,x)=[E_j^s(x,d)]^\top$
for all $x\in G$, $d\in\Sp^2$ and $j=1,2$ (see \cite[Theorem 6.31]{CK}), we obtain that
$E_1^s(x,d)=E_2^s(x,d)$ for all $x\in G$ and all $d\in\Sp^2$
or $E_1^\infty(\hat{x},d)=E_2^\infty(\hat{x},d)$ for all $\hat{x},\;d\in\Sp^2$.
By the uniqueness result for inverse electromagnetic scattering with full
far-field data (see \cite[Theorem 7.1]{CK} for the obstacle case and \cite[Theorem 4.9]{Hahner}
for the inhomogeneous medium case) it follows easily that the uniqueness statements (a) and (b)
of Theorem \ref{ele_main} are true. The theorem is thus proved.
\hfill$\Box$

\section{Conclusions}\label{con}
\setcounter{equation}{0}

This paper proposed a new approach to prove uniqueness results for inverse acoustic and electromagnetic 
scattering for obstacles and inhomogeneous media with phaseless near-field data at a fixed frequency.
The idea is to use superpositions of two point sources at a fixed frequency as the incident fields
and, as the phaseless near-field data, to measure the modulus of the acoustic total-field on two 
spheres enclosing the scatterers generated by such incident fields located on the two spheres,
in the acoustic case. For the electromagnetic case, the idea is to utilize superpositions of two 
electric dipoles at a fixed frequency with the polarization vectors $\bm e_\phi$ and $\bm e_\theta$, 
respectively, as the incident fields and, as the phaseless near-field data, to measure the modulus 
of the tangential component with the orientations $\bm e_\phi$ and $\bm e_\theta$, respectively, 
of the electric total-field on a sphere enclosing the scatterers and generated by such incident fields 
located on the measurement sphere and another bigger sphere. 
As far as we know, this is the first uniqueness result for three-dimensional inverse electromagnetic
scattering with phaseless near-field data.

\section*{Acknowledgements}

The authors thank Professor Xudong Chen at the National University of Singapore for helpful and
constructive discussions on the measurement technique of electromagnetic waves.
This work is partly supported by the NNSF of China grants 91630309 and 11871466.

\end{document}